\documentclass[english,letterpaper,11pt,reqno]{amsart}
\usepackage{appendix}
\usepackage{amsbsy}
\usepackage{amsfonts}
\usepackage{amsmath}
\usepackage{amssymb}
\usepackage{amsthm}
\usepackage{graphicx}
\usepackage{ifthen}
\usepackage{textcomp}
\usepackage{enumitem, color, amssymb}
\usepackage{mathrsfs}
\usepackage[bookmarksnumbered,colorlinks]{hyperref}
\hypersetup{colorlinks=true, linkcolor=blue, citecolor=blue}
\usepackage{placeins}

\newtheorem {lemma} {Lemma} [section]
\newtheorem{thm}{Theorem}
\newtheorem{theorem}{Theorem}

\newtheorem{prop}[lemma]{Proposition}

\newtheorem{cor}[lemma]{Corollary}
\theoremstyle{remark}

\newcounter{nmdthmcnt}

\newcommand{\beqa}{\begin{eqnarray}}
\newcommand{\beq}{\begin{equation}}
\newcommand{\eeqa}{\end{eqnarray}}
\newcommand{\eeq}{\end{equation}}

\newcommand{\be}{\begin{equation}}
\newcommand{\ee}{\end{equation}}
\newcommand{\lb}[1]{\label{#1}}

\renewcommand{\Ref}[1]{(\ref{#1})}

\newcommand\kk{{\bf k}}

\newcommand\xx{{\bf x}}
\newcommand\yy{{\bf y}}
\newcommand\kf{\hat\kk}
\newcommand\tf{\hat\tT}
\newcommand\xf{\hat\xx}
\newcommand\yf{\hat\yy}
\newcommand\kv{\kk}
\newcommand\tv{\tT}
\newcommand\xv{\xx}
\newcommand\yv{\yy}

\newcommand\n{\nabla}

\newcommand{\HH}{\mathcal{H}}
\newcommand{\VV}{\mathcal{V}}
\newcommand{\UU}{\mathcal{U}}
\newcommand{\ta}{\tau}
\newcommand\om{\omega}
\newcommand\tT{{\bf t}}

\newcommand{\al}{\alpha}
\newcommand{\bet}{\beta}

\newcommand{\we}{\wedge}

\newcommand{\sig}{\sigma}

\newcommand{\lam}{\lambda}
\newcommand{\fr}{\frac}

\begin{document}
\title[]{On the completeness of some Bianchi type A and related K\"ahler-Einstein metrics}
\author[]{Gideon Maschler and Robert Ream}
%\address{Dept. of Mathematics, Clark University, Worcester, MA 01610}
%\email{Aaazami@clarku.edu\,,\,Gmaschler@clarku.edu}
\address{Department of Mathematics and Computer Science\\ Clark University\\ Worcester, MA }
\email{Gmaschler@clarku.edu\,,\,Rream@clarku.edu}
%\email{gmaschler@clarku.edu}

\maketitle
\thispagestyle{empty}
\begin{abstract}
We prove the existence of complete cohomogeneity one triaxial K\"ahler-Einstein metrics
in dimension four under an action of the Euclidean group $E(2)$. We also demonstrate local existence of Ricci flat K\"ahler metrics of a related type that are given via
generalized PDEs, and determine, under mild conditions, whether they are complete. The common framework for both metric types is a frame-dependent system of Lie bracket relations and generalized PDEs yielding a class of K\"ahler-Einstein metrics on $4$-manifolds which includes all diagonal Bianchi type A metrics.
\end{abstract}

%-----------------------
\section{Introduction}

The study of complete curvature-distinguished metrics which are invariant under a cohomogeneity one
group action have seen significant progress in recent decades, cf. \cite{dw1}, \cite{dw2}, \cite{b}, \cite{bdw}, \cite{dhw}, \cite{w}. Much of this research centers on the case of compact
groups. A significant portion of this paper concerns non-compact groups. See \cite{f} for other work involving non-compact groups.

The class of examples we study includes four-dimensional unimodular cohomogeneity one K\"ahler-Einstein metrics, otherwise known as Bianchi type A K\"ahler-Einstein metrics. This class of metrics has been studied by mathematicians and physicists, see \cite{bgpp}, \cite{p}, \cite{d-s1}, \cite{d-s2}.

In the last two references, Dancer and Strachan prove that when the cohomogeneity one action
is under the group $SU(2)$, there exist so-called triaxial complete cohomogeneity one K\"ahler-Einstein metrics. No similar results appear to be known for non-compact unimodular groups. In one of the main results given in the second part of this paper, we demonstrate the existence of complete triaxial cohomogeneity one K\"ahler-Einstein metrics under the action of the Euclidean plane group $E(2)$.
\begin{theorem}
A cohomogeneity one $E(2)$ $4$-manifold with a discrete principal stabilizer admits complete K\"ahler-Einstein metrics, and all such solutions are classified.
\end{theorem}
See Theorem~\ref{E2thm} for a precise statement.
One part of the proof proceeds along similar lines to those in \cite{d-s1} for $SU(2)$. Another part is systematized via recent results of Verdiani and Ziller \cite{vz}. The final completeness argument involves Cauchy-Schwarz type estimates, a method that has been employed recently in \cite{a-m2} to prove completeness in the simpler case of biaxial metrics for a quotient of the Heisenberg group.

In the first part of this work, which is more foundational, we give a framework for the study
of the existence problem for a class of K\"ahler-Einstein metrics in dimension four
that includes the diagonal Bianchi type A class.
%(but also some non-diagonal metrics).
Our framework is inspired by the references \cite{a-m}, \cite{a-m2}, but does not share
their Lorentzian geometric motivation.

Concretely, we describe a class of K\"ahler metrics via a set of Lie bracket
relations on a given local frame, and impose the K\"ahler-Einstein condition
as a set of generalized PDEs involving directional derivatives in the frame
directions. See \cite{a-m2} for a precise definition of a generalized PDE.

For the case where the Einstein constant is nonzero, we show that the resulting
system gives rise locally to a system of ODEs. The Bianchi type A metrics form
examples for this case. On the other hand, in one of the Ricci flat subcases, an analogous process yields a small number of generalized PDEs with directional derivatives along a two-dimensional distribution.

These two derivations appear in the appendix, and in the rest of the paper we investigate each of these two cases. We have already mentioned the completeness result pertaining to the first case. For the Ricci flat subcase, there are additional well-known incomplete cohomogeneity one examples, but the fact that this case is characterized by PDEs rather than ODEs raises the question of whether there are additional examples. We show that there are no {\em complete} examples of a certain type, but prove local existence of additional solutions.
\begin{theorem}
There exists a K\"ahler Ricci-flat metric $g$ on an open
neighborhood of $\mathbb{R}^4$ admitting a totally-geodesic two-dimensional
foliation whose leaves are isometric to a given model leaf-metric.
The dimension of the Lie algebra of Killing fields of $g$ is at least $2$.
\end{theorem}
See Theorem~\ref{g-bar-g} for a precise statement.

Our method in the Ricci flat subcase is as follows. First, such a metric must possess a two-dimensional totally geodesic foliation. Aided by a theorem due to Yau \cite{y}, we show
that if the leaves are finitely connected and the leaf metric is not flat, it cannot be complete or even possess a completion obtained by adjoining one point to the leaf. This automatically
proves incompleteness of the K\"ahler metric with a foliation of this type. We then show by an explicit construction that in dimension two, there exist metrics which possess the same properties
as a leaf metric must have. The main such property is that the metric has positive Gauss curvature
which, when used as a conformal factor, yields a metric with Gauss curvature $-2$.  Then we show that using these leaf-like metrics as local data, constructing the Ricci-flat K\"ahler metric locally amounts to a collection of decoupled PDE problems, which are handled by elementary means.

Returning to our framework, the presentation of the problem in terms of a frame with Lie bracket relations can be considered dual to one given in terms of an exterior differential system
(EDS). We note that a proof of the reduction of the problem in the non-Ricci-flat case
to an ODE system via EDS methods would involve four prolongations to reach involutivity.

As a means for simplifying standard EDS methods for the purpose of constructing solutions, this reduction procedure is not particularly remarkable. But we find it greatly aids in maintaining geometric intuition for the problem. For example, one byproduct in the K\"ahler case is that integrability of the complex structure has a particularly appealing form in terms of conditions on
the shears of the frame vector fields. Additionally, it is fairly straightforward to recast geometric problems in this form. In future work we plan to study the case of Ricci solitons for a similar system.

Section~\ref{sec:sh} contains preliminary material on shear and integrability of complex
structures. Sections~\ref{construct}-\ref{sec:KE} introduce the system for K\"ahler-Einstein
metrics, and Section~\ref{OdE} describes how to construct for it an ODE/PDE system as described
above, with the proof in the appendix. In section~\ref{sec:coho} we specialize to diagonal
cohomogeneity one metrics, and consider simple examples in sections~\ref{ric-flat}-\ref{sec:Heis}.
Section~\ref{sec:Euc} is devoted to the case of the Euclidean group $E(2)$, and culminates in the
proof of completeness of the appropriate K\"ahler-Einstein metrics. Section~\ref{generalized} is
devoted to the incompleteness results in the Ricci-flat subcase. In section~\ref{coords}
local existence is shown first for the leaf metrics, and from that for the corresponding
Ricci-flat K\"ahler metrics.

%----------------------
\section{Shear and integrability}\lb{sec:sh}

Let $(M,g,J)$ be an almost hermitian $4$-manifold. We fix a local oriented orthonormal frame
denoted
\[
\{e_i\}=\{\kk, \tT=J\kk, \xx, \yy=J\xx\}.
\]
In the frame domain, we have an orthogonal decomposition of the tangent bundle:
\[
\text{$TM=\VV\oplus\HH$, \quad with ${\VV}=\mathrm{span}(\kk,\tT),\quad {\HH}=\mathrm{span}(\xx,\yy)$.}
\]
Let $\UU$ stand for either $\VV$ or $\HH$, and $\pi_{\UU^\perp}:TM\to{\UU^\perp}$ denote the orthogonal projection.
For a vector field $X\in\Gamma(\UU)$, consider the operator $\pi_{\UU^\perp}\circ \n X|_{\UU^\perp}:\Gamma(\UU^\perp)\to\Gamma({\UU^\perp})$, where $\n$ is the Levi-Civita covariant
derivative of $g$. Define the {\em shear operator} of $X$ by
\[
\text{$S_X$\ :=\ trace-free symmetric part of $\pi_{\UU\perp}\circ \n X|_{\UU^\perp}$.}
\]
See \cite{a-m} for background on the relation to the shear operator in general relativity.

Our purpose here is to give a condition equivalent to the integrability of $J$ in terms of
shear operators.
\begin{thm}\lb{Nij0}
Given the above set-up, the almost complex structure $J$ is integrable in the frame domain
if and only if
\begin{align}\lb{Nij}
\mathrm{i})&\ \ J S_{\xx}=S_{\yy} \text{ on ${\VV}$.}\nonumber\\
\mathrm{ii})&\ \ J S_{\kk}=S_{\tT} \text{ on ${\HH}$.}
%\mathrm{and}\quad \mathrm{ii})\ J\n^o\kk_+=\n^o\!J\kk_+ \text{ on ${\HH}$.}
\end{align}
\end{thm}
The proof is similar to \cite[Theorem 1]{a-m}, and will be omitted. Its method
is to translate the vanishing of the Nijenhuis tensor into the above two conditions.
This relies on the following expression of the matrix corresponding to the
shear operator in a local oriented orthonormal frame $\{v_1,v_2\}$ on $\UU^\perp$.
\[
[S_X]_{v_1,v_2}=
\begin{bmatrix}%{cc}
        - \sigma_1 & \sigma_2\\
        \sigma_2 & \sigma_1\\
      \end{bmatrix},
\]
with {\em shear coefficients}:
\be\lb{sh-coef}
\begin{aligned}
2\sigma_1\ &:=\
% g(\n_\yy X,\yy) - g(\n_\xx X,\xx)=
 \ \ g([X,v_1],v_1)-g([X,v_2],v_2),\\
2\sigma_2\ &:=\
% g(\n_\yy X,\xx) + g(\n_\xx X,\yy)=
 -g([X,v_1],v_2) - g([X,v_2],v_1).
\end{aligned}
\end{equation}

One simple case in which integrability holds by Theorem \ref{Nij0}
is when all the shears vanish: $S_{e_i}=0$, $i=1,\ldots,4$.
We refer to this as the shear-free case. Our main results concern
cases which are not shear-free.

%-----------------------
\section{Shear and K\"ahler metrics}\lb{construct}

Let $(M,g)$ be a Riemannian $4$-manifold admitting an orthonormal frame $\{e_i\}=\{\kk, \tT, \xx, \yy\}$,
defined over an open $U\subset M$, which satisfies the Lie bracket relations
%\be\lb{brack}
\begin{align}
&[\kk,\tT]=L(\kk+\tT),\qquad &&[\xx,\yy]=N(\kk+\tT),\lb{brack1}\\
&[\kk,\xx]=A\xx+B\yy,\qquad  &&[\kk,\yy]=C\xx+D\yy,\lb{brack2}\\
&[\tT,\xx]=E\xx+F\yy,\qquad  &&[\tT,\yy]=G\xx+H\yy,\lb{brack3}
\end{align}
%\end{equation}
for smooth functions $A, B, C, D, E, F, G, H, L, N$ on $U$ such that
%\be\lb{rels}
\begin{align}
A&-D=F+G,\qquad B+C=H-E,\lb{rels1}\\
N&=A+D=-(E+H).\lb{rels2}
\end{align}
%\end{equation}
%Assume the frame $\kk,\tT,\xx,\yy$ is orthonormal with respect
%to a Riemannian metric $g$, and
Define an almost complex structure
$J=J_{g,e_i}$ by linearly extending the relations $J\kk=\tT$, $J\tT=-\kk$, $J\xx=\yy$
and $J\yy=-\xx$.
\begin{prop}\lb{kah}
(M,g,J) defined as above is a K\"ahler structure on $U$.
\end{prop}
\begin{proof}
$J$ clearly makes $g$ into an almost hermitian metric.
To see that $J$ is integrable, we verify the conditions of Theorem~\ref{Nij0}
which, in view of \Ref{sh-coef}, are expressed in terms of shear coefficients
as $\sig_1^\yy=\sig_2^\xx$, $\sig_2^\yy=-\sig_1^\xx$, $\sig_1^\tT=\sig_2^\kk$,
$\sig_2^\tT=-\sig_1^\kk$. Using \Ref{sh-coef} and the orthonormality of our frame,
by \Ref{brack2}, \Ref{brack3}, the first pair of these
equations each takes the form $0=0$, whereas the second pair is equivalent to the assumed
relations \Ref{rels1}.

%Then $J$ is of the type
%considered in \cite{a-m}, while
%the last two lines of \Ref{brack} along with the first line
%of \Ref{rels} guarantee that the conditions of Theorem~1 in
%\cite{a-m} hold, so that $J$ is integrable.
%We will mostly be interested in the case where at least one
%of the equations in the first line of \Ref{rels}
%is not of the form $0=0$.
To show that $g$ is K\"ahler,
define a connection on $U$ by first setting
\be\lb{conn}
\n_\kk\kk=-L\tT,\qquad  \n_\xx\xx=A\kk+E\tT,\qquad
\n_\xx\kk=-A\xx+E\yy,%\qquad &&\n_\kk\xx=(H-C)\yy.
\end{equation}
and then having all other covariant derivative expressions on frame fields
determined by the requirement that $\n$ be torsion-free and make $J$ parallel
(here the definition of $J$ and relations \Ref{brack1}-\Ref{rels2} are used
repeatedly).
It is easily checked that $\n$ is compatible with the metric $g$, so that
it is its Levi-Civita connection and thus $J$ is $g$-parallel. This completes the proof.
\end{proof}

\section{The Ricci form}

The Ricci form of the K\"ahler metric $g$ in Proposition~\ref{kah} is computed
as follows. Denote by $w_1=\kk-i\tT$, $w_2=\xx-i\yy$ the corresponding complex-valued frame,
and compute the four complex valued $1$-forms $\Gamma_i^j, i,j=1,2$ for which
$\n w_i=\Gamma_i^j\otimes w_j$, where here $\n$ denotes the obvious complexification
of the Levi-Civita connection of $g$. The formulas are deduced by computing the components
$\n_{e_\ell} w_i$, where $e_\ell$ stands for one of the frame fields, using the covariant derivative frame formulas for the Levi-Civita connection $\n$, given in the proof of Proposition~\ref{kah}.
Two of the four $\Gamma_i^j$'s resulting from this calculation are
\[
\Gamma_1^1=-iL(\hat\kk+\hat\tT),\qquad \Gamma_2^2=-i(C-H)\hat\kk-i(A-F)\hat\tT,
\]
where the hatted quantities denote  the non-metrically-dual coframe of $\{e_\ell\}$.
Citing, for example, Lemma 4.2 in \cite{dr-ma}, the Ricci form of $g$ is given by
\begin{multline}\lb{ric}
\rho=i(d\Gamma_1^1+d\Gamma_2^2)=L(d\hat\kk+d\hat\tT)+(C-H)d\hat\kk+(A-F)d\hat\tT\\
+dL\we(\hat\kk+\hat\tT)+d(C-H)\we\hat\kk+d(A-F)\we\hat\tT.
\end{multline}

%------------------------------
\section{The K\"ahler-Einstein condition}\lb{sec:KE}
Suppose the metric $g$ of Sec.~\ref{construct} is
K\"ahler-Einstein, so that
\[
\rho=\lam\om,
\]
where $\om=\hat\kk\we\hat\tT+\hat\xx\we\hat\yy$ is the K\"ahler form, and $\lam$ is
the Einstein constant. We wish to rewrite this equation in a different form.
Applying to our coframe the formula
$d\eta(a,b)=d_a(\eta(b))-d_b(\eta(a))-\eta([a,b])$,
valid for any smooth $1$-form $\eta$, we have
\[
\begin{aligned}
&d\hat\kk(\xx,\yy)=-\hat\kk([\xx,\yy])=-\hat\kk (N(\kk+\tT))=-N=d\hat\tT(\xx,\yy),\\
&d\hat\kk(\kk,\tT)=-L=d\hat\tT(\kk,\tT),\qquad d\hat\kk(\kk,\xx)=d\hat\kk(\kk,\yy)=
d\hat\kk(\tT,\xx)=d\hat\kk(\tT,\yy)=0.
\end{aligned}
\]
Using this in \Ref{ric} along with the expression for $\om$ in the coframe, we
immediately see that the K\"ahler-Einstein equation is equivalent
to the system
%This is equivalent
%to the following six equations, obtained by equating the Ricci form and
%$\lam$ times the Kahler form on the frame fields.
\be\lb{KE-eqns}
\begin{aligned}
&\rho(\xx,\yy)=-N(2L+C-H+A-F)=\lam,\\
&\rho(\kk,\tT)=-L(2L+C-H+A-F)+d_{\kk-\tT}L-d_\tT(C-H)+d_\kk(A-F)=\lam,\\
&\rho(\kk,\xx)=-d_\xx(L+C-H)=0,\\
&\rho(\kk,\yy)=-d_\yy(L+C-H)=0,\\
&\rho(\tT,\xx)=-d_\xx(L+A-F)=0,\\
&\rho(\tT,\yy)=-d_\yy(L+A-F)=0,
\end{aligned}
\end{equation}
where $d_{e_\ell}$ denotes the directional derivative with respect to $e_\ell$.
As equations like~\Ref{KE-eqns} involve directional derivatives rather than partial
derivatives, they were called generalized PDEs in \cite{a-m2}, where a
precise definition appears.

\section{The ODE and generalized PDE systems}\lb{OdE}

In the appendix we prove Theorem~\ref{ODE-thm} below, showing that Lie bracket conditions
\Ref{brack1}-\Ref{rels2}, together with the generalized PDEs \Ref{KE-eqns} characterizing the
K\"ahler-Einstein condition, determine when $\lam\ne 0$, a locally defined system of five ODEs on six functions, while in one of the $\lam=0$ subcases they give rise to a system of four
generalized PDEs along a two-dimensional distribution. Our purpose in this section is to introduce these two systems.

We first describe a function that will eventually give rise to the independent variable
for the ODE system we are trying to obtain. Assuming the setting of section \ref{construct},
the Lie bracket relations \Ref{brack1}-\Ref{brack3} imply that the distribution spanned by
$\kk+\tT$, $\xx$ and $\yy$ is integrable. Since this distribution is orthogonal to $\kk-\tT$,
while the latter vector field has constant length and is easily seen to have geodesic flow, it
follows that it is locally a gradient (cf. \cite[Cor. 12.33]{onel}). Thus, there exists
a smooth function $\ta$ defined in some open set $V\subset U$, such that
\be\lb{grad}
\kk-\tT=\n\ta.
\end{equation}
Consider now the six functions $P$, $Q$, $R$, $S$, $L$, $N$,
where the last two are as in \Ref{brack1}, and the first four
are given in terms of four of the functions in \Ref{brack2}-\Ref{brack3} by
\begin{align}
P&=(B-C)+(F-G), &&Q=(B-C)-(F-G),\nonumber \\
R&=\sqrt{(B+C)^2+(F+G)^2}, &&S=\tan^{-1}\left(\frac{B+C}{F+G}\right),\lb{chan-var}
%-\tan^{-1}k,  \text{ for a constant $k$.}
\end{align}
where $S$ is only defined on the set $\{F+G\}\ne 0$.
\begin{thm}\lb{ODE-thm}
Let $(M,g)$ be a Riemannian 4-manifold admitting an orthonormal frame
$\{\kk,\tT,\xx,\yy\}$ as in section 3, satisfying in particular
relations \Ref{brack1}-\Ref{rels2}. Assume also that \Ref{KE-eqns} hold in
the set $V\cap\{F+G\ne 0\}$, where $V$ is the domain of $\ta$ of \Ref{grad}.
\begin{enumerate}
\item Assume $\lam\ne 0$.\\
%\begin{align*}
%&\text{$\lam\ne 0$, or}&&\\
%&\text{$\lam=0$ and $2L+C-H+A-F\ne 0$ on $V\cap\{F+G\ne 0\}$.}
%&\ &&\text{$2L+C-H+A-F=0$ and $L\ne 0$ on $V\cap\{F+G\ne 0\}$.}
%\end{align*}
Then $L$, $N$, $P$, $Q$, $R$, $S$  above are each a composition of a smooth real-valued function on the image of $\ta$. Additionally, abusing notation by still denoting the latter functions
%defined on the image of $\ta$
by the same respective letters as the former, they satisfy
on $V\cap\{F+G\ne 0\}$ the ODE system
\begin{align}
N'&=N^2-LN,   &&L'=L^2-N^2+NP/4+R^2/4,\nonumber\\
R'&=(P/2+L)R, &&P'=PL+R^2,\nonumber\\
S'&=-Q/2.&&   \lb{sys}
%2\lambda&=-N(4L+2N-P) &&
\end{align}
\item Assume $\lam=0$ and $N=0$.\\
Then  $P$, $Q$, $R$, $S$, $L$ satisfy on $V\cap\{F+G\ne 0\}$
the system
\begin{align}
i&)\ d_{\kk-\tT}R=R(d_{\kk+\tT}S+P+2L),  &&ii)\ d_{\kk+\tT}R=-R(d_{\kk-\tT}S+Q),\nonumber\\
iii&)\ d_{\kk-\tT}L=2L^2+R^2/2, &&iv)\ d_{\kk-\tT}P-d_{\kk+\tT}Q=2LP+2R^2,\lb{sys2a}
\end{align}
whereas all $d_\xx$, $d_\yy$ derivatives of these functions vanish.
\end{enumerate}
\end{thm}

We make the following remarks. If $F+G=0$  at $p$, but $B+C\ne 0$ at $p$, one can obtain
similar systems of equations valid at $p$ simply by redefining  $S$ to be $S+\tan^{-1} k$ for a constant $k$. Thus $R\ne 0$ is the only invariant restriction on the domain, corresponding to considering the non-shear-free region where the shear operators $S_\kk$ and $S_\tT$ do not vanish.

Second, and relatedly, note that the last equation in \Ref{sys} is decoupled from the others, so that one has the freedom to arbitrarily choose, say, $S$. Since $S$ is determined by $B+C$ and $F+G$, both of which appear in \Ref{rels1}, this is a reflection of the fact that the shear coefficients are not invariant, and a rotation of, $\xx$, $\yy$ in the plane they span will alter them, giving them, and hence $S$, arbitrary values, without changing the metric under consideration. One can similarly use such a rotation to simplify the form of equation \Ref{sys2a}, effectively eliminating in this case one of the variables $S$, $P$ and $Q$. We will employ such a choice in section \ref{generalized}.

One further point regarding case (1) of the theorem is that
if $\lam\ne 0$,
%If $N$ is nowhere vanishing,
the second equation in \Ref{sys}, for $L'$, can be replaced by the constraint
\be
2\lambda=-N(4L+2N-P),
\end{equation}
which allows one to eliminate one of the functions $L$, $N$, $P$, reducing the number of unknowns.
%On the other hand, in the $\lam=0$ case, it will follow that $N=0$ in $V\cap \{F+G\ne 0\}$,
%so that the first equation in \Ref{sys} is an identity, while the second one reduces to
%\be\lb{n=0}
%L'=L^2+R^2/4.
%\end{equation}
%Finally, in the second $\lam=0$ case, $P=4L+2N$ which also allows one to give the second equation
%in \Ref{sys} another form.
See the appendix for the proof.

In the next three sections we discuss a large class of examples satisfying \Ref{brack1}-\Ref{rels2}
and \Ref{KE-eqns}, under both assumptions regarding $\lambda$ and $N$ given in Theorem~\ref{ODE-thm}. We will not be employing case $(1)$ of this theorem in those sections,
% actual Theorem~\ref{ODE-thm} there,
as other presentations of the equations seem more amenable to exploring various issues, in particular completeness of the metrics. We expect equations \Ref{sys} to be utilized in a future study of new types of metrics satisfying the system \Ref{brack1}-\Ref{rels2} and \Ref{KE-eqns}.
On the other hand, in sections~\ref{generalized} and \ref{coords} we rely heavily on a version
of equations \Ref{sys2a} of case (2) of the theorem.

\section{Cohomogeneity one examples}\lb{sec:coho}
In this section we begin the second part of this paper. We first discuss a notable class of examples of K\"ahler-Einstein metrics on $4$-manifolds admitting a frame satisfying conditions \Ref{brack1}-\Ref{rels2}.

Assume that $(M,g)$ is a $4$-dimensional Riemannian manifold admitting a proper isometric action by a Lie group $\mathcal{G}$ with cohomogeneity one.
Then there is a subgroup $\mathcal{H}<\mathcal{G}$ so that $\mathcal{G}/\mathcal{H}$ is the $3$-dimensional principal orbit type.
Let $p\in M$ be a point with isotropy group $\mathcal{K}$ satisfying $\mathcal{H}<\mathcal{K}<\mathcal{G}$. Then the orbit $\mathcal{G}\cdot p$ through $p$ is isomorphic to $\mathcal{G}/\mathcal{K}$. For the principal $\mathcal{K}$-bundle $\mathcal{G}\to\mathcal{G}/\mathcal{K}$,
consider the associated bundle $\mathcal{G}\times_{\mathcal{K}}\nu_p$, where $\nu_p$ is
the normal space to the orbit at $p$. The differential of the action mapping identifies this bundle
with the full normal bundle $\nu$ to the orbit. On the other hand the normal exponential map $\exp_p^\perp$ at $p$ sends an $\varepsilon$-disk in $\nu_p$ to a
%An $\varepsilon$-neighborhood of $p$ in the directions normal to the orbit $\mathcal{G}\cdot p$ gives a
slice for the action of $\mathcal{G}$
\[ S' = \{\exp_p(rX)\ |\ 0\le r < \varepsilon, |X|=1, X\perp \mathcal{G}\cdot p \}. \]
and induces, by the tubular neighborhood theorem, a map from a neighborhood of the
zero section in $\nu$ to a neighborhood of the orbit. Putting these facts together
we obtain an equivariant diffeomorphism,
\[ \mathcal{G}\times_\mathcal{K} D^{n+1}\cong\mathcal{G}\cdot S', \]
where $n=\mathrm{dim}\,\mathcal{K}-\mathrm{dim}\,\mathcal{H}$ (cf. \cite[Section 5.6]{pp}).

The isotropy action of $\mathcal{K}$ preserves length, so on $S'$, we see that the spheres
\[ S_r = \{\exp_p(rX)\ |\ |X|=1, X\perp \mathcal{G}\cdot p \} \]
are preserved by the induced action of $\mathcal{K}$.
Since points on one of these spheres have isotropy type $\mathcal{H}$,
we must have \[ \mathcal{K}/\mathcal{H} \cong \mathbb{S}^n. \]

Regarding the metric as residing on $\mathcal{G}\times_\mathcal{K} D^{n+1}$, it can be written
in the form
\begin{equation}\label{equivmet} dr^2+g_r.\end{equation}
%In the next subsection we will match this presentation of the metric
%with one given via an appropriate orthonormal coframe as in the previous sections. For now, we
%remark that the relation $\tau=\sqrt{2}r$ will hold.
We will assume from now on that $\mathcal{G}$ has dimension $3$ with $\mathcal{H}$ a discrete
principal stabilizer.
In the case of a unimodular group, we will now consider the special case of a diagonal metric, in the
form appearing,  for example, in \cite{d-s1}.
When $\mathcal{G}$ is unimodular, $g$ has Bianchi type A, and can be written as
\begin{equation}\label{bianchiAmet} g = (abc)^2dt^2+a^2\sigma_1^2+b^2\sigma_2^2+c^2\sigma_3^2, \end{equation}
for functions $a$, $b$, $c$ of $t$ and invariant $1$-forms $\sig_1$, $\sig_2$, $\sig_3$.
%We will see this implies that $d\tau=\sqrt{2}abcdt$.
This change of variables allows a more effective use of standard ODE theory, as
will become evident in later sections.
%shown that it is justified in the relevant cases.

The $\sigma_i$ satisfy, for some constants $p_i$
\begin{align*}
d\sigma_1 &= p_1\sigma_2\wedge\sigma_3, \\
d\sigma_2 &= p_2\sigma_3\wedge\sigma_1, \\
d\sigma_3 &= p_3\sigma_1\wedge\sigma_2.
\end{align*}
In terms of the basis $\partial_t,X_1,X_2,X_3$ dual to $dt,\sigma_1,\sigma_2,\sigma_3$
\begin{align}
[\partial_t,X_i] &= 0,\quad i=1,\ldots,3,\nonumber  \\
[X_1,X_2] &= -p_3 X_3,\nonumber \\
[X_2,X_3] &= -p_1 X_1, \nonumber\\
[X_3,X_1] &= -p_2 X_2.\lb{X123}
\end{align}
For the functions $w_1=bc$, $w_2=ac$, and $w_3=ab$, define functions $\alpha$, $\beta$, and $\gamma$ so that
\begin{align}
w_1'&=p_1w_2w_3+\alpha w_1, \\
w_2'&=p_2w_1w_3+\beta w_2, \\
w_3'&=p_3w_1w_2+\gamma w_3.
\end{align}
Then, following Dancer and Strachan \cite{d-s1}, we see that (modulo reordering the frame vectors) the only K\"ahler structures $(M,g,J)$ with $g$ of the form \Ref{bianchiAmet} have complex structure determined by
\begin{equation}\label{complexJ} J\partial_t = abX_3 \quad\text{and}\quad JX_1=\frac{a}{b}X_2, \end{equation}
and $\alpha$, $\beta$, and $\gamma$ satisfy
\[ \alpha=\beta \quad \text{and}\quad \gamma=0.\]
The K\"ahler form is then given by
\begin{equation}\label{kahlerform} \omega = abc^2dt\wedge\sigma_3+ab\sigma_1\wedge\sigma_2 = w_1w_2dt\wedge\sigma_3+w_3\sigma_1\wedge\sigma_2, \end{equation}
and $w_1,w_2,w_3$ satisfy
\begin{align}
w_1'&=p_1w_2w_3+\alpha w_1,\nonumber \\
w_2'&=p_2w_1w_3+\alpha w_2,\nonumber \\
w_3'&=p_3w_1w_2.\lb{w3}
\end{align}
In terms of $a,b,c$ this implies
\begin{align}
2a'/a&=-p_1a^2+p_2b^2+p_3c^2,\lb{K1} \\
2b'/b&=p_1a^2-p_2b^2+p_3c^2, \lb{K2}\\
2c'/c&=p_1a^2+p_2b^2-p_3c^2+2\alpha.\lb{KKE}
\end{align}

In the next subsection we derive the Einstein condition, after showing how
this model fits within the framework of sections~\ref{construct}-\ref{sec:KE}.

\subsection{The frame $\{\mathbf{k,t,x,y}\}$}

In this subsection we show how the metric $g$ of the previous subsection
gives rise to data satisfying \Ref{brack1}-\Ref{rels2} and \Ref{KE-eqns}.
Consider the orthonormal frame and dual coframe
\begin{align*}
\mathbf{k} &= \frac{\sqrt{2}}{2}\left(\frac{1}{c}X_3+\frac{1}{abc}\partial_t \right), & \hat{\mathbf{k}} &= \frac{\sqrt{2}}{2}(c\sigma_3+abcdt), \\
\mathbf{t} &= \frac{\sqrt{2}}{2}\left(\frac{1}{c}X_3-\frac{1}{abc}\partial_t \right), & \hat{\mathbf{t}} &= \frac{\sqrt{2}}{2}(c\sigma_3-abcdt), \\
\mathbf{x} &= \frac{X_1}{a}, & \hat{\mathbf{x}} &= a\sigma_1, \\
\mathbf{y} &= \frac{X_2}{b}, & \hat{\mathbf{y}} &= b\sigma_2.
\end{align*}

It can easily be checked that this frame satisfies \Ref{brack1}-\Ref{brack3}
for the functions
%Then the variables $A,B,C,D,E,F,G,H,L,N$ are
\begin{align*}
A&=-E=-\frac{a'}{\sqrt{2}a^2bc}=-\frac{1}{a}\frac{da}{d\tau}, & B &=F= -\frac{bp_2}{\sqrt{2}ac}, \\
D&=-H=-\frac{b'}{\sqrt{2}ab^2c}=-\frac{1}{b}\frac{db}{d\tau}, & C&=G=\frac{ap_1}{\sqrt{2}bc}, \\
L&=-\frac{c'}{\sqrt{2}abc^2}=-\frac{1}{c}\frac{dc}{d\tau}, & N &= -\frac{cp_3}{\sqrt{2}ab}.
\end{align*}

Here the prime denotes differentiation with respect to $t$, while the expressions in terms of
$d/d\ta$ are justified as follows. We know that relations \Ref{brack1}-\Ref{brack3} imply that
$\ta$ is locally defined and $d\tau=\hat{\mathbf{k}}-\hat{\mathbf{t}}$. Our metric
can be written as
\begin{equation}\label{framemet} g=\hat{\mathbf{k}}^2+\hat{\mathbf{t}}^2+\hat{\mathbf{x}}^2+
\hat{\mathbf{y}}^2=\frac{1}{2}d\tau^2+\frac{1}{2}(\hat{\mathbf{k}}+
\hat{\mathbf{t}})^2+\hat{\mathbf{x}}^2+\hat{\mathbf{y}}^2. \end{equation}
This is a special case of \Ref{equivmet} for $\tau=\sqrt{2}r$.
Furthermore,
\begin{align*}
\frac{\mathbf{k}-\mathbf{t}}{\sqrt{2}} &= \frac{\partial_t}{abc} & \frac{\hat{\mathbf{k}}-\hat{\mathbf{t}}}{\sqrt{2}} &= abcdt, \\
\frac{\mathbf{k}+\mathbf{t}}{\sqrt{2}} &= \frac{X_3}{c} & \frac{\hat{\mathbf{k}}+\hat{\mathbf{t}}}{\sqrt{2}} &= c\sigma_3, \\
\end{align*}
and
\[ \hat{\mathbf{k}}-\hat{\mathbf{t}}=d\tau=\sqrt{2}abcdt,\]
so that
\[ \frac{d}{d\tau}=\frac{1}{\sqrt{2}abc}\frac{d}{dt}. \]

\vspace{0.1in}
We note that for our list of functions $A$,\ldots,$H$, $L$, $N$, the four relations in \Ref{rels1}-\Ref{rels2} that imply the K\"ahler condition
impose only two additional relations here, say $A+D=N$ and $B+C=H-E$,
giving
\begin{align}
\fr{a'}a+\fr{b'}b&=p_3c^2,\lb{K3}\\
\fr{b'}b-\fr{a'}a&=p_1a^2-p_2b^2\lb{another}
\end{align}
These two are of course equivalent to \Ref{K1}-\Ref{K2}.

Next we consider the K\"ahler-Einstein equations \Ref{KE-eqns}.
The last four are satisfied automatically by virtue of the fact that
our $10$ functions are functions of ($t$, hence) $\ta$, and $d_\xx\ta=d_\yy\ta=0$.

Now the second of equations \Ref{KE-eqns} can be written in the form
\be\lb{K4}
- L(2L + C - H + A - F) + \fr d{d\ta}(2L + C - H + A - F) =\lam.
\end{equation}
To proceed further, consider the case $p_3\ne 0$.
Then $N$ is nowhere vanishing, and using the first equation
in \Ref{KE-eqns}, we can replace \Ref{K4} by the equation for $\fr d{d\ta}N$ in
\Ref{sys}. Checking, we easily find that the latter equation is equivalent
to \Ref{K3}. Thus if $p_3\ne 0$, the only additional independent
equation characterizing the K\"ahler-Einstein condition
is the first equation in \Ref{KE-eqns}, which, after simplifying
and using \Ref{K3} takes the form
\[
2\fr{c'}c=p_1a^2+p_2b^2-p_3c^2-\fr{2\lam}{p_3} (ab)^2.
\]
Comparing with \Ref{KKE} we deduce the relation
$\alpha =-\frac{\lambda}{p_3}a^2b^2$.

In the case where $p_3=0$, we have $N=0$. The first equation in \Ref{KE-eqns}
gives $\lam=0$. The second equation in \Ref{KE-eqns} is \Ref{K4},
which reads, via \Ref{KKE},
\begin{align*}
\fr {c'}{2(abc)^2}(-2\al)&=-\fr1{\sqrt{2}abc}\Big[\fr1{\sqrt{2}abc}(-2\al)\Big]'\\
&=-\fr1{\sqrt{2}abc}\fr1{\sqrt{2}ab}\fr{-c'}{c^2}(-2\al)-\fr {1}{2(abc)^2}(-2\al'),
\end{align*}
where we have used the fact that $ab$ is constant, as follows from \Ref{w3} since $p_3=0$.
As the left hand side is equal to the first term on the right, we see that we must have $\al'=0$.

To summarize, the K\"ahler conditions \Ref{rels1}-\Ref{rels2} are expressed as \Ref{K3}-\Ref{another},
while the Einstein condition can be summarized as
\begin{equation} p_3\neq0, \alpha = -\frac{\lambda}{p_3}w_3^2=-\frac{\lambda}{p_3}a^2b^2 \qquad\text{or}\qquad p_3=0,\lambda=0,\alpha'=0. \end{equation}

Note that besides $p_3=0$, the condition $\al=0$ also yields $\lam=0$.
This is in line with Theorem~\ref{ODE-thm}, as $2L+C-H+A-F$ is a multiple of $\al$.

\vspace{.1in}
Finally, the functions $P,Q,R,S$ are given below for completeness, as they will
not be used further.
\begin{align*}
P&=-\sqrt{2}\frac{a^2p_1+b^2p_2}{abc}, & Q &= 0, \\
R&=\frac{a^2p_1-b^2p_2}{abc}, & S &= \frac{\pi}{4}.
\end{align*}

\section{Ricci Flat Metrics with $p_3=0$}\lb{ric-flat}
The next two sections describe some examples of the metrics in the previous section.

When $p_3=0$ the system can be solved explicitly.  In terms of $w_1$, $w_2$, and $w_3$ we have
\begin{align*}
w_1'&=p_1w_2w_3+\alpha w_1, \\
w_2'&=p_2w_1w_3+\alpha w_2, \\
w_3'&=0,\quad
\alpha'=0.
\end{align*}
This can be written as
\begin{align*}
(e^{-\alpha t}w_1)'&=p_1w_3e^{-\alpha t}w_2, \\
(e^{-\alpha t}w_2)'&=p_2w_3e^{-\alpha t}w_1,
\end{align*}
where $\alpha$ and $w_3$ are constant.
This implies that
\[(e^{-\alpha t}w_1)''=p_1 p_2 w_3^2(e^{-\alpha t}w_1).\]
The solution then splits into four cases: $p_1p_2<0$, two cases with $p_1p_2=0$, and $p_1p_2>0$.

\subsubsection*{Case 1: Poincar\'e Group $p_1 p_2<0$}
Assume that $p_1=1$ and $p_2=-1$, then we have
\[(e^{-\alpha t}w_1)''=-w_3^2(e^{-\alpha t}w_1).\]
The solution is
\begin{align*}
w_1 &= ke^{\alpha(t-t_0)}\sin(w_3(t-t_0)),\\
w_2 &= ke^{\alpha(t-t_0)}\cos(w_3(t-t_0)).
\end{align*}
Therefore
\begin{align*}
a&= \sqrt{w_3 \cot(w_3(t-t_0))},\\
b&= \sqrt{w_3 \tan(w_3(t-t_0))},\\
c&=ke^{\alpha (t-t_0)}\sqrt{\frac{1}{2w_3}\sin(2w_3(t-t_0))},
\end{align*}
and the metric is
\begin{align*}g&=\frac{k^2}{2}e^{2\alpha (t-t_0)}\sin(2w_3(t-t_0))w_3\left(dt^2+\frac{\sigma_3^2}{w_3^2}\right)\\
&\qquad+w_3 \cot(w_3(t-t_0))\sigma_1^2+w_3 \tan(w_3(t-t_0))\sigma_2^2.\end{align*}
The singular points $t=t_0,t_0+\frac{\pi}{2w_3}$ are both at finite distance, so this metric is not complete.

\subsubsection*{Case 2: Abelian Group $T^3$ $p_1=p_2=0$}

Here the solution is
\begin{align*}
a&=a_0,\\
b&=b_0,\\
c&=c_0e^{\al(t-t_0)}
\end{align*}
The metric takes the form
\[ g=\Big(\fr{a_0b_0c_0}\al\Big)^2\Bigg(\Big(de^{\al(t-t_0)}\Big)^2+
\Big(\fr{\al}{a_0b_0}\Big)^2e^{2\al(t-t_0)}\,\sig_3^2\Bigg)+a_0^2\,\sig_1^2+b_0^2\,\sig_2^2
\]
This is a product of $T^2$ with a two dimensional cone metric,
which gives a smooth flat metric when  $\al=a_0b_0$.

\subsubsection*{Case 3: Heisenberg Group $p_1 p_2=0$}
Assume that $p_1=1$ and $p_2=0$.
Then the solution is
\begin{align*}
w_1&=kw_3e^{\alpha (t-t_0)}(t-t_0), \\
w_2&=k e^{\alpha (t-t_0)}.
\end{align*}
Therefore
\begin{align*}
a&=\frac{1}{\sqrt{t-t_0}}, \\
b&=w_3\sqrt{t-t_0}, \\
c&=ke^{\alpha (t-t_0)}\sqrt{t-t_0},
\end{align*}
and the metric is
\[g=k^2w_3^2e^{2\alpha(t-t_0)}(t-t_0)dt^2+\frac{1}{t-t_0}\sigma_1^2+w_3^2(t-t_0)\sigma_2^2+k^2e^{2\alpha(t-t_0)}(t-t_0)\sigma_3^2.\]
The singular point $t=t_0$ is at finite distance, so this metric is not complete.

\subsubsection*{Case 4: Euclidean Group $p_1 p_2>0$}
Assume that $p_1=1$ and $p_2=1$, then we have
\[(e^{-\alpha t}w_1)''=w_3^2(e^{-\alpha t}w_1).\]
The solution is
\begin{align*}
w_1 &= ke^{\alpha(t-t_0)}\sinh(w_3(t-t_0)),\\
w_2 &= ke^{\alpha(t-t_0)}\cosh(w_3(t-t_0)),
\end{align*}
Therefore
\begin{align*}
a&= \sqrt{w_3 \coth(w_3(t-t_0))},\\
b&= \sqrt{w_3 \tanh(w_3(t-t_0))},\\
c&=ke^{\alpha (t-t_0)}\sqrt{\frac{1}{2w_3}\sinh(2w_3(t-t_0))},
\end{align*}
and the metric is
\begin{align*}g&=\frac{k^2}{2}e^{2\alpha (t-t_0)}\sinh(2w_3(t-t_0))w_3\left(dt^2+\frac{\sigma_3^2}{w_3^2}\right)\\
&\qquad+w_3 \coth(w_3(t-t_0))\sigma_1^2+w_3 \tanh(w_3(t-t_0))\sigma_2^2.\end{align*}
The singular point $t=t_0$ is at finite distance, so this metric is not complete.

The Ricci-flat metrics with $p_3\neq 0$, but  $p_1p_2=0$ are similar to those with $p_3=0$ and $\al=0$.
The Ricci-flat metrics with $p_1p_2p_3\neq 0$ are addressed in \cite{bgpp}.

\section{The Heisenberg Group with $p_3\ne 0$}\lb{sec:Heis}
When $p_1=p_2=0$, $p_3=1$, and $\lambda=-1$ the Lie-algebra spanned by $X_1,X_2,X_3$ is the nilpotent Heisenberg Lie-algebra.
The K\"ahler-Einstein equations are
\begin{align*}
2a'/a&=c^2, \\
2b'/b&=c^2, \\
2c'/c&=-c^2+2a^2b^2.
\end{align*}
This has two first integrals
\begin{align*}
\left(\frac{a}{b}\right)'&=0,\\
\left(ab\left(c^2-\frac{2}{3}a^2b^2\right)\right)'&=0.
\end{align*}
The first one implies that $a$ is a constant multiple of $b$.
A metric with such a property was termed {\em biaxial} in \cite{d-s1},
and {\em triaxial} if it does not have this property. In terms of
shears, the biaxial case occurs in the shear-free case, and
is generally simpler from the point of view of integrability
of the solutions. It was shown in \cite{a-m2} that there is
a complete K\"ahler-Einstein metric of this type on a manifold
admitting a cohomogeneity one action of a quotient of the Heisenberg group
by a discrete subgroup. The metric was described in that reference
via an orthonormal frame as in section~\ref{construct}.

For this reason, we will not go any further into the description
of this metric. Of the remaining Bianchi type A cases with $p_3\ne 0$,
the case of $SU(2)$ has been extensively studied. Of the remaining
cases involving the noncompact groups $SL(2,\mathbb{R})$, the Poincare
group and the Euclidean group, we describe complete metrics only
for the latter, in the next section.

\section{The Euclidean Group with $p_3\ne 0$}\lb{sec:Euc}
In this section we describe a complete triaxial K\"ahler-Einstein
metric with a cohomogeneity one action of the Euclidean group $E(2)$.
The method follows in part the work in \cite{d-s1} and \cite{d-s2}
which gives an analogous result for the case of the compact group
$SU(2)$. Two main differences in method from those references are
a systematic use of recent results of Verdiani and Ziller \cite{vz},
and the establishment of Cauchy-Schwarz type estimates yielding
completeness for this non-compact group.

We set $p_2=0$, $p_1=p_3=1$, and $\lambda=-1$.
The Lie-algebra spanned by $X_1,X_2,X_3$ is the Lie-algebra of the Euclidean group.
The K\"ahler-Einstein equations are
\begin{align}
\label{E2ODEa}2a'/a&=-a^2+c^2, \\
\label{E2ODEb}2b'/b&=a^2+c^2, \\
\label{E2ODEc}2c'/c&=a^2-c^2+2a^2b^2
\end{align}
Note that the derivatives in this system are given by polynomials in the dependent variables,
hence are locally Lipschitz, so that standard ODE theory applies. As for the symmetries
of these equations, first, they are autonomous, so constant shifts in $t$ preserve
solutions. Finally, the scaling symmetry
\[
(a(t),b(t),c(t))\to (ka(k^2t),b(k^2t),kc(k^2t)),
\]
taking solutions to solutions, will play a role later on.

\subsection{Linearization about Equilibria}
The non-zero equilibrium solutions are $(q,0,q)$ and $(0,q,0)$.
Then in terms of $a$, $b$, and $c$,
\begin{align*}
a'&=\frac{a}{2}(-a^2+c^2), \\
b'&=\frac{b}{2}(a^2+c^2), \\
c'&=\frac{c}{2}(a^2-c^2+2a^2b^2).
\end{align*}
has linearization about $(q,0,q)$
\begin{align*}
a'&=-q^2a+q^2c, \\
b'&=q^2b, \\
c'&=q^2a-q^2c.
\end{align*}
which has one positive, one negative and one zero eigenvalue. The linearization about $(0,q,0)$ has three zero eigenvalues.

\begin{thm}\lb{E2thm}
A solution of \Ref{E2ODEa}-\Ref{E2ODEc} yields a complete K\"ahler-Einstein metric
of the form \Ref{bianchiAmet} on a cohomogeneity one $E(2)$-manifold if and only if it
is a solution along an unstable curve of an equilibrium point $(q,0,q)$, $q>0$.
\end{thm}
\begin{proof}
The proof is broken into three steps.
First, in Proposition~\ref{prop:solutions} we show that solutions with a maximal interval
having a finite left endpoint do not give complete metrics. Then we show that solutions with maximal interval of the form $(-\infty,\eta)$ are the unstable curves
of the equilibrium points $(q,0,q)$, $q>0$, and these solutions satisfy $0\le c^2-a^2 \le 2a^2b^2$.
Next, in Proposition~\ref{prop:estimates} we show that $\eta$ is finite,
but for curves on the manifold orthogonal to the orbits the endpoint corresponding to $\eta$ is infinitely far away, while the endpoint corresponding to $-\infty$ is at finite distance.
To complete the proof that this gives us a complete K\"ahler-Einstein metric we show that $g$ and $\omega$ can be extended smoothly as $t\to-\infty$, and then finish the proof of completeness by
showing all finite length curves remain inside some compact set.
These are completed in Propositions~\ref{prop:bolt} and \ref{prop:complete}, respectively.
\end{proof}

We first record in a lemma some relations,  easily verifiable via \Ref{E2ODEa}-\Ref{E2ODEc}, which will
be used later in the proof.
%We give the lemma in a form valid for any unimodular group.
\begin{lemma}\lb{prelim}
For the system \Ref{E2ODEa}-\Ref{E2ODEc},
\begin{align*}
(ab)'&=abc^2, \\
(bc)'&=bca^2\left(1+b^2\right), \\
(ac)'&=a^3cb^2, \\
\left(\frac{a}{b}\right)'&=-\frac{a^3}{b}, \\
(a^2)'&=(a^2)(-a^2+c^2) \\
(a^2-c^2)'&=-(a^2-c^2)(a^2+c^2)-2 a^2b^2c^2 \\
-(c^2)'&=-c^2(a^2-c^2)-2 a^2b^2c^2
\end{align*}
%\begin{align*}
%(ab)'&=p_3abc^2, \\
%(bc)'&=bca^2\left(p_1-\frac{\lambda}{p_3}b^2\right), \\
%(ac)'&=acb^2\left(p_2-\frac{\lambda}{p_3}a^2\right), \\
%\left(\frac{a}{b}\right)'&=-\frac{a}{b}(p_1a^2-p_2b^2), \\
%(p_1a^2-p_2b^2)'&=(p_1a^2-p_2b^2)(-p_1a^2-p_2b^2+p_3c^2) \\
%(p_1a^2-p_3c^2)'&=(p_1a^2-p_3c^2)(-p_1a^2+p_2b^2-p_3c^2)+2\lambda a^2b^2c^2 \\
%(p_2b^2-p_3c^2)'&=(p_2b^2-p_3c^2)(p_1a^2-p_2b^2-p_3c^2)+2\lambda a^2b^2c^2
%\end{align*}
\end{lemma}

\subsection{Solutions}
\begin{prop}\label{prop:solutions}
There are no complete metrics corresponding to solutions of \Ref{E2ODEa}-\Ref{E2ODEc} with maximal interval $(\xi,\eta)$, when $\xi$ is finite. Furthermore, the only solutions with maximal interval $(-\infty,\eta)$ are the unstable curves of the equilibrium points $(q,0,q)$, $q>0$, and these solutions satisfy $0\le c^2-a^2\le 2a^2b^2$.
\end{prop}
\begin{proof}
For an initial time $t_0$, let $(\xi,\eta)$ be a maximal solution interval for the initial value problem for \Ref{E2ODEa}-\Ref{E2ODEc} with
$a(t_0)=a_0$, $b(t_0)=b_0$, and $c(t_0)=c_0$.

Uniqueness of solutions to \Ref{E2ODEa}-\Ref{E2ODEc} implies that if any of
$a$, $b$, or $c$ are zero anywhere in $(\xi,\eta)$ then it is zero everywhere.
Accordingly we assume that $a$, $b$, and $c$ are all positive on $(\xi,\eta)$.
Then we see from Lemma~\ref{prelim} and \Ref{E2ODEb} that $ab$, $bc$, $ac$, and $b$ are all
increasing on $(\xi,\eta)$.

We consider the following cases:
\subsubsection*{Case 1: $c_0^2-a_0^2<0$} We first make the following claim.\\[4pt]
Claim: In this case $a\to\infty$ as $t\to\xi^+$.\\[3pt]
\textit{Proof of claim:} Since
\[ (c^2-a^2)'=-(c^2-a^2)(c^2+a^2)+2a^2b^2c^2, \]
if $c^2-a^2<0$ then $(c^2-a^2)'>0$, thus $c^2-a^2<0$ for all $\xi<t<t_0$.
Therefore,
\begin{align*}
a' &= \frac{a}{2}(c^2-a^2), \\
a'' &= \frac{a}{4}[(c^2-a^2)^2-2(c^2-a^2)(c^2+a^2)+4a^2b^2c^2],
\end{align*}
showing that $a$ is decreasing and concave up on $(\xi,t_0)$.
Next, we always have $b'>0$, while
\[c'=\frac{c}{2}(a^2-c^2+2a^2b^2)>0, \]
shows that $c$ is increasing on $(\xi,t_0)$.
Therefore $b$ and $c$ are bounded on $(\xi,t_0)$.
Thus, as $(\xi,\eta)$ is the maximal solution interval,
$a$ could be bounded as $t\to\xi^+$ only if $\xi=-\infty$.
But since $a$ is concave up, $a\to\infty$ as $t\to\xi^+$ even when $\xi=-\infty$. \qed

Since $ab$ and $ac$ are increasing, they are bounded as $t\to\xi^+$ and $a\to\infty$, so $b\to0$, $c\to0$, and $ab\to k$ for some constant $k$.
Then as $t\to\xi^+$ the equations will take the asymptotic form
\begin{align*}
a'&=-\frac{1}{2}a^3 \\
b'&=\frac{1}{2}ba^2 \\
c'&=\frac{1}{2}c(a^2+2k^2) \\
\end{align*}
the solution of which has asymptotic form
\begin{align*}
a&\simeq (t-\xi)^{-\frac{1}{2}},\\
b&\simeq b_1(t-\xi)^{\frac{1}{2}},\\
c&\simeq c_1(t-\xi)^{\frac{1}{2}},\\
\end{align*}
for some constants $b_1$ and $c_1$.
This shows that $\xi$ is finite in this case and
\[ \int_{\xi}^{t_0} abc\,dt <\infty,\]
so the metric is not complete.

\subsubsection*{Case 2: $c_0^2-a_0^2>2a_0^2b_0^2$}Here we have a similar claim.\\[4pt]
Claim: In this case $c\to\infty$ as $t\to\xi^+$.\\[3pt]
\textit{Proof of claim:} Since
\[ (c^2-a^2-2a^2b^2)'=-(c^2-a^2)(c^2+a^2)-2a^2b^2c^2, \]
if $c^2-a^2-2a^2b^2>0$ then $(c^2-a^2-2a^2b^2)'<0$, thus in this case, $c^2-a^2>2a^2b^2$ for all $\xi<t\le t_0$.
Therefore,
\begin{align*}
c' &= \frac{c}{2}(a^2-c^2+2a^2b^2), \\
c'' &= \frac{c}{4}[(a^2-c^2+2a^2b^2)^2+2(c^2-a^2)(c^2+a^2)+4a^2b^2c^2],
\end{align*}
showing that $c$ is decreasing and concave up on $(\xi,t_0)$.
Next, we always have $b'>0$, while
$a$ is increasing on $(\xi,t_0)$.
Therefore $a$ and $b$ are bounded on $(\xi,t_0)$.
Thus, as $(\xi,\eta)$ is the maximal solution interval,
$c$ could be bounded as $t\to\xi^+$ only if $\xi=-\infty$.
But since $c$ is concave up, $c\to\infty$ as $t\to\xi^+$
even when $\xi=-\infty$. \qed

Since $ac$ and $bc$ are increasing, they are bounded as $t\to\xi^+$ and $c\to\infty$, so $a\to0$, $b\to0$.
Then as $t\to\xi^+$ the equations will take the asymptotic form
\begin{align*}
a'&=\frac{1}{2}ac^2 \\
b'&=\frac{1}{2}bc^2 \\
c'&=-\frac{1}{2}c^3 \\
\end{align*}
which has solution
\begin{align*}
a&\simeq a_1(t-\xi)^{\frac{1}{2}}\\
b&\simeq b_1(t-\xi)^{\frac{1}{2}}\\
c&\simeq (t-\xi)^{-\frac{1}{2}}\\
\end{align*}
for some constants $a_1$ and $b_1$.
This shows that $\xi$ is finite in this case and
\[ \int_{\xi}^{t_0} abc\,dt <\infty,\]
so the metric is not complete.

If $c^2-a^2<0$ or $c^2-a^2>2a^2b^2$ at any time, then a constant shift in $t$ will give one of the previous cases.
In both previous cases, $\xi$ is finite, but we know that the unstable curve of the equilibrium points $(q,0,q)$ must have $\xi=-\infty$.
The existence of these curves is guaranteed by the center manifold theorem.
Therefore we consider the final case:

\subsubsection*{Case 3: $0\le c^2-a^2 \le 2a^2b^2\textnormal{ for all }t\in(\xi,\eta)$} Here we
have a different claim.\\[4pt]
Claim: In this case $\xi=-\infty$. \\[3pt]
\textit{Proof of claim:}
In this case $a$, $b$, and $c$ are all increasing, therefore they are all bounded on $(\xi,t_0)$.
Since $(\xi,\eta)$ is the maximal solution interval $\xi=-\infty$. \qed

As $a$, $b$, and $c$ are all increasing, it must be that they all approach finite non-negative limits as $t\to-\infty$.
Thus $(a,b,c)$ must approach an equilibrium point.
If $(a,b,c)\to(0,q,0)$ with $q>0$, then $a/b\to0$ as $t\to-\infty$, but $a/b$ is decreasing and positive
(see Lemma~\ref{prelim}), so this cannot happen.

Therefore, when $t\to-\infty$ we see that $(a,b,c)\to(q,0,q)$ in this case, but we still need to rule out
the possibility that $q=0$. For this we compute the variation of $ac$ with respect to $b$:
%notice that whenever $(a(t),b(t),c(t))$ is a solution of (\ref{E2ODEa})-(\ref{E2ODEc}), so is the scaled %curve $(ka(k^2t),b(k^2t),kc(k^2t))$.
%This suggest finding the dependence of $(a/c)$ and $b$.
%Using $b$ as an independent variable we obtain the ODE
\begin{align}
\frac{d(ac)}{db}&=\frac{2(\frac{a}{c})^2(ac)b}{(\frac{a}{c})^2+1}.\lb{ac-b}
\end{align}
Our assumption of an equilibrium point $(q,0,q)$ implies that $a/c\to 1$ when $b\to 0$.
%From the first of these equations we see that the slope field is negative along $\frac{a}{c}=1$, so
Since in our case $c^2-a^2\ge 0$, we have
\[ \frac{a}{c}\le 1. \]
Employing this in equation \Ref{ac-b} yields
\[
\frac{d(ac)}{db}\le (ac)b.
\]
By Gr\"onwall's inequality, if $ac\to 0$ when $b\to 0$ then $ac=0$ identically. As the latter
is not possible (see the beginning of the proof), neither is $q=0$.
\end{proof}

%\begin{prop}\label{prop:estimates}
%Let $g$ be a metric of the form \Ref{bianchiAmet} on a manifold $M$, with $a$, $b$, $c$ a solution to %\Ref{E2ODEa}-\Ref{E2ODEc} along an unstable curve of an equilibrium point $(q,0,q)$, $q>0$.
%For any curve in $M$ $g$-orthogonal to the $G$-orbits containing a point $p$,
%% following curves along $t$,
%its endpoint corresponding to $\xi=-\infty$ is at finite distance, while its endpoint corresponding to %$\eta$ is at infinite distance, from $p$.
%\end{prop}
\begin{prop}\label{prop:estimates}
Let $g$ be a Riemannian metric of the form \Ref{bianchiAmet} on a manifold $M$, with $a$, $b$, $c$ a solution to \Ref{E2ODEa}-\Ref{E2ODEc} along an unstable curve of an equilibrium point $(q,0,q)$, $q>0$,
having maximal domain $I=(-\infty,\eta)$. Assume that the latter interval is also the range of the coordinate function $t$ on $M$.
For a point $p_0\in M$ with orbit through $p_0$ of principal type and $M^t$ a level set of $t$,
\[
\lim_{t\to -\infty}d_g(p_0,M^t)<\infty,\qquad
\lim_{t\to \eta}d_g(p_0,M^t)=\infty,
\]
where $d_g$ is the distance function induced by $g$.
%Suppose $\gamma: I\to M$ is a smooth injective map defining a
%curve in $M$ which passes through a point $p_0$, is $g$-orthogonal to the $G$-orbits and satisfies
%$t\circ\gamma=\mathrm{id}_I$.
%Then
%\[
%\lim_{t\to -\infty}d_g(p_0,\gamma(t))<\infty,\qquad
%\lim_{t\to \eta}d_g(p_0,\gamma(t))=\infty,
%\]
%where $d_g$ is the distance function induced by $g$.
\end{prop}
\begin{proof}
The union of the principal orbits forms an open dense set, $\tilde{M}$, so that
$\tilde{M}/\mathcal{G}$ is a smooth manifold of dimension 1. The function $t$ is
a smooth submersion from $\tilde{M}$ to $\tilde{M}/\mathcal{G}$. The metric
\[ (abc)^2dt^2\]
makes this into a Riemannian submersion.
The level sets of $t$ are orbits of $\mathcal{G}$ and for $t_0=t(p_0)$
\[d_g(p_0,M^{t_1}) = d_g(M^{t_0},M^{t_1}), \]
is the distance in the quotient manifold where
\[ d_g(M^{t_0},M^{t_1}) = \left|\int_{t_0}^{t_1} abc dt\right|. \]

Asymptotically as $t\to-\infty$,
\begin{align*}
a &\simeq q\\
b &\simeq ke^{q^2t}\\
c &\simeq q
\end{align*}
This gives the asymptotic metric
\[ g\simeq k^2q^4e^{2q^2t}dt^2+q^2\sigma_1^2+k^2e^{2q^2t}\sigma_2^2+q^2\sigma_3^2, \]
and for $v=ke^{q^2t}$ this is just
\[ g\simeq (dv^2+v^2\sigma_2^2)+q^2(\sigma_1^2+\sigma_3^2). \]
In this coordinate, the endpoint $\xi=-\infty$ is at $v=0$, and we see that
\[ \lim_{t\to -\infty}d_g(p_0,M^t) = \int_{-\infty}^{t_0} abc dt = \int_0^{v_0} dv < \infty. \]

Now to understand the behavior at the $\eta$ side of the solution interval,
%notice that whenever $(a(t),b(t),c(t))$ is a solution of (\ref{E2ODEa})-(\ref{E2ODEc}), so is the scaled %curve $(ka(k^2t),b(k^2t),kc(k^2t))$.
%This suggest finding the dependence of $(a/c)$ and $b$.
we examine the derivative of $a/c$ with respect to $b$
\be\lb{a/c-b}
\frac{d(\frac{a}{c})}{db}=
\frac{2(\frac{a}{c})}{b}\left(\frac{1-(\frac{a}{c})^2(1+b^2)}{(\frac{a}{c})^2+1}\right),
\end{equation}
%Using $b$ as an independent variable we obtain the ODE
%\[\frac{d(\frac{a}{c})}{db}=\frac{2(\frac{a}{c})}{b}\left(\frac{1-(\frac{a}{c})^2(1+b^2)}{(\frac{a}{c})^2+1}\right). %\]
This equation has nullcline
\[\frac{a}{c}=\sqrt{\frac{1}{1+b^2}},\]
Since the nullcline is always decreasing, and our solution starts at $\frac{a}{c}=1$ when $b=0$, we have
\[ \frac{a}{c}\ge\sqrt{\frac{1}{1+b^2}}. \]
%Recall also that we have already seen
%Likewise, the slope field is negative along $\frac{a}{c}=1$ so
%\[ \frac{a}{c}\le 1. \]

To find a better upper bound than $1$ for $a/c$, we consider the curve 
$\frac ac=\sqrt{\fr{k^2}{k^2+b^2}}$ and plug its expression into the slope
field. This gives slope
\[
\fr{2k^2+2b^2}{2k^2+b^2}(k^2-1)\fr d{db}\sqrt{\fr{k^2}{k^2+b^2}}.
\]
So for $k^2>2$, the slope of the solutions along this curve are less, i.e. more negative,
than the slope of the curve. Since for our (non-equilibrium) solution, there is some $b_p$
in the range of $b$ for which one has $(\fr ac)(b_p)<1$, the graph of $a/c$
is below the graph of $\sqrt{\frac{k^2}{k^2+b^2}}$ for some $k=k_p>2$ at $b_p$, and hence for
all $b\ge b_p$. Therefore for $b\ge b_p$
%\[ \sqrt{\frac{1}{1+b^2}}\le \frac{a}{c}\le\sqrt{\frac{2}{2+b^2}}. \]
\[ \sqrt{\frac{1}{1+b^2}}\le \frac{a}{c}\le\sqrt{\frac{k_p^2}{k_p^2+b^2}} \]
Next we deduce an estimate for $ac$ in terms of $b$, for $b>b_p$. Using \Ref{ac-b} and the last inequalities, we have
%\[ \frac{2b}{2+b^2} \le \frac{d(\ln(ac)}{db} \le \frac{4b}{4+b^2}, \]
\[ \frac{2b}{2+b^2} \le \frac{d(\ln(ac)}{db} \le \frac{2bk_p^2}{2k_p^2+b^2}, \]
and integrating this from $b=0$ to $b$, exponentiating and multiplying by $q^2$ gives
%\[ \ln\frac{2+b^2}{2} \le \ln\frac{ac}{q^2} \le 2\ln\frac{4+b^2}{4}. \]
\[ \ln\frac{2+b^2}{2} \le \ln\frac{ac}{q^2} \le k_p^2\ln\frac{2k_p^2+b^2}{2k_p^2}. \]
Therefore
%\[ q^2\frac{2+b^2}{2} \le ac \le q^2\left(\frac{4+b^2}{4}\right)^2. \]
\[ q^2\frac{2+b^2}{2} \le ac \le q^2\left(\frac{2k_p^2+b^2}{2k_p^2}\right)^{k_p^2}. \]
Finally, these can be used to estimate $b$ as
\[ b' = \frac{b}{2}(a^2+c^2) =\frac{b}{2}(ac)\left(\frac{a}{c}+\frac{c}{a}\right).\]
Therefore, for some positive constant $K_1$ and $b = b(t) > b_p$,
\[ b'\ge K_1b^4,\]
showing that $\eta$ is finite, but for some constant $K_2>0$, using the notation
$b_0 := max\{b(t_0),b_p \} > 0$, we have
\begin{equation}\label{eta_dist} \lim_{t\to \eta}d_g(p_0,M^t)= \int_{t_0}^\eta abc\, dt \ge\int_{b_0}^\infty \frac{2}{\frac{a}{c}+\frac{c}{a}}\,db \ge \int_{b_0}^\infty\frac{K_2}{b} db=\infty. \end{equation}
\end{proof}

\subsection{The Bolt}
The phrase ``attaching a bolt" refers to replacing a $4$-manifold with a cohomogeneity one action
with only regular fibers over an open interval with one admitting a similar action for the same group
over a semi-closed interval with a two dimensional singular fiber (the bolt) over the endpoint of the interval.
For the case at hand, the latter $4$-manifold can be described as
\[
E(2)\times_{SO(2)}\mathbb{R}^2= (0,\infty)\times E(2)\ \amalg\ \{0\}\times \mathbb{R}^2 ,
\]
where the right $SO(2)$-action is $(g, (T,x))\to (Tg,g^{-1}x)$.
\begin{prop}\label{prop:bolt}
The metric and K\"ahler form corresponding to solutions of \Ref{E2ODEa}-\Ref{E2ODEc} along the unstable curves of the equilibrium points $(q,0,q)$, $q>0$, defined on $(-\infty, \eta)$,
can be smoothly extended to $E(2)\times_{SO(2)}\mathbb{R}^2$, with the bolt fibering over $\xi=-\infty$.
\end{prop}
\begin{proof}
Consider the manifold $M=E(2)\times_{SO(2)}\mathbb{R}^2$.  This has a left action by $E(2)$ with regular orbit $E(2)$ and singular orbit $E(2)/SO(2)$.
For any $E(2)$ invariant metric $g$ on $M$, with $r$ the distance along a geodesic perpendicular to the singular orbit,
\[ g = dr^2+g_r. \]
For a metric $g$ of the form \Ref{bianchiAmet}, let $r=\int_{-\infty}^t a(s)b(s)c(s)\,ds$, then
\[ g = dr^2+a^2\sigma_1^2+b^2\sigma_2^2+c^2\sigma_3^2. \]
The ODE's \Ref{E2ODEa}-\Ref{E2ODEc} in this coordinate become
\begin{align}
\label{rODEa}\frac{da}{dr}&=\frac{a}{2}\left(-\frac{a}{bc}+\frac{c}{ab}\right),\\
\label{rODEb}\frac{db}{dr}&=\frac{1}{2}\left(\frac{a}{c}+\frac{c}{a}\right),\\
\label{rODEc}\frac{dc}{dr}&=\frac{c}{2}\left(\frac{a}{bc}-\frac{c}{ab}+2\frac{ab}{c}\right).
\end{align}
From these it is seen that $a,b,$ and $c$ can be extended at $r=0$ so that $a$ and $c$ are even and $b$ is odd.
Following the notations of Verdiani and Ziller \cite{vz}, the tangent space for $r\neq 0$ splits as
\[ T_p M = \mathbb{R}\partial_r\oplus \mathfrak{k}\oplus\mathfrak{m}, \]
where
\[\mathfrak{k}=\mathrm{span}\{X_2\}, \]
\[ \mathfrak{m}=\mathrm{span}\{X_1,X_3\}=:\ell_1, \]
and we set
\[ V=\mathrm{span}\{\partial_r,X_2\}=:\ell_{-1}'. \]
Since $\exp(\theta X_2)$ acts on both $V$ and $\mathfrak{m}$ as a rotation by $\theta$, we have weights $a_1=d_1=1$.
The smoothness conditions for $V$ is that $b$ can be extended to an odd function and $b'(0)=1$.
Since we know that $b$ can be extended to be odd, we just check from \Ref{rODEb} that
\[ \left.\frac{db}{dr}\right|_{r=0}=\frac{1}{2}\left(\frac{q}{q}+\frac{q}{q}\right)=1.\]
Since $\ell'_{-1}$ and $\ell_1$ are perpendicular, the smoothness conditions in table C of
\cite{vz} are automatically satisfied, while those in
table B there, are
\begin{align}\label{smooth1}a^2+c^2&=\phi_1(r^2),\\
\label{smooth2}a^2-c^2&=r^2\phi_2(r^2),
\end{align}
for some smooth functions $\phi_1$ and $\phi_2$.
Now to see that \Ref{smooth1} is satisfied, note that
\[ a^2+c^2=2ac\frac{db}{dr}.\]
Since $a$, $c$, and $\frac{db}{dr}$ are even, it just remains to check \Ref{smooth2}.
As $a$, $c$ are even while $b$ is odd, $a/c$ is even as a function of $b$.
We have to second order
\[ a^2-c^2=c^2\left(\frac{a^2}{c^2}-1\right)\propto b^2c^2, \]
and since $b(0)=0$ and $\frac{db}{dr}|_{r=0}=1$, we get that $g$ extends to a smooth metric on $M$.

Finally we check that the K\"ahler form extends across the singular orbit at $r=0$.
Following Verdiani and Ziller we analyze the eigenspaces for the action of $SO(2)$ on $T_pM$.
We find that $\partial_r+\frac{i}{r}X_2$ is an eigenvector with eigenvalue $e^{ia_1\theta}$, and
$X_1+iX_3$ is an eigenvector with eigenvalue $e^{id_1\theta}$, and likewise for their complex conjugates.
Dualizing gives eigenspaces of $T^*_pM$: $dr-ir \sigma_2$ has eigenvalue $e^{ia_1\theta}$, and $\sigma_1-i\sigma_3$ has eigenvalue  $e^{id_1\theta}$.
Thus the eigenspaces of $\Lambda^2T^*_pM$ are
\begin{align*}
E_1 &= \mathrm{span}\{r dr\wedge \sigma_2, \sigma_1\wedge\sigma_3\}\\
E_{e^{i(a_1-d_1)\theta}}&=\mathrm{span}\{dr\wedge\sigma_1+r\sigma_2\wedge\sigma_3+i(dr\wedge\sigma_3+r\sigma_1\wedge\sigma_2)\}\\
E_{e^{i(a_1+d_1)\theta}}&=\mathrm{span}\{dr\wedge\sigma_1-r\sigma_2\wedge\sigma_3+i(dr\wedge\sigma_3-r\sigma_1\wedge\sigma_2)\}\\
&\qquad
\end{align*}
The smoothness condition is just the equivariance condition $\omega(e^{a_1\theta}p)=\exp(\theta X_2)^*\omega$.
This requires that the coefficient of
\begin{align}
\label{eigen1} E_1 & \text{ is } \phi_1(r^2), \\
\label{eigen2} E_{e^{\pm i(a_1-d_1)\theta}}& \text{ is } r^{\frac{|a_1-d_1|}{a_1}}\phi_2(r^2),\\
\label{eigen3} E_{e^{\pm i(a_1+d_1)\theta}}& \text{ is } r^{\frac{|a_1+d_1|}{a_1}}\phi_3(r^2).
\end{align}
Now we have
\begin{align*}\omega&=cdr\wedge\sigma_3+ab\sigma_1\wedge\sigma_3\\
&=\frac{c}{2}[(dr\wedge\sigma_3+r\sigma_1\wedge\sigma_2)+(dr\wedge\sigma_3-r\sigma_1\wedge\sigma_2)]\\
&\qquad + \frac{ab}{2r}[(dr\wedge\sigma_3+r\sigma_1\wedge\sigma_2)-(dr\wedge\sigma_3-r\sigma_1\wedge\sigma_2)]\\
&=\frac{cr+ab}{2r}(dr\wedge\sigma_3+r\sigma_1\wedge\sigma_2) \\
&\qquad + \frac{cr-ab}{2r}(dr\wedge\sigma_3-r\sigma_1\wedge\sigma_2).
\end{align*}
Using $a_1=d_1=1$ in \Ref{eigen2}-\Ref{eigen3}, the smoothness conditions can now be written as
\begin{align*}
cr+ab &= r\phi_2(r^2), \\
cr-ab &= r^3\phi_3(r^2). \\
\end{align*}
The first of these is clear; for the second, expand to get $a/c= 1+\al b^2+\mathcal{O}(b^4)$
for some constant $\al$ and $b= r+\mathcal{O}(r^3)$, so
\[ cr\left(1-\frac{ab}{cr}\right)=cr\left(1-\frac{b+\al b^3
+\mathcal{O}(b^5)}{r}\right)=r^3\phi_3(r^2).\]
Therefore $\omega$ extends as a smooth form on all of $M$.
\end{proof}

\subsection{Completeness}
\begin{prop}\label{prop:complete}
For the metrics of Proposition~\ref{prop:bolt},
all finite length curves remain inside some compact set.
\end{prop}
\begin{proof}
For the Euclidean group
\[E(2)=\left\{ \left( \begin{array}{ccc} \cos\theta & -\sin\theta & x \\ \sin\theta & \cos\theta & y \\ 0 & 0 & 1 \end{array}\right)\ \Bigg|\ x,y,\theta\in\mathbb{R}\right\}, \]
The left-invariant frame
\[ X_1=\cos\theta\partial_x+\sin\theta\partial_y,\, X_2=\partial_\theta, \, X_3=-\sin\theta\partial_x+\cos\theta\partial_y, \]
has dual co-frame
\[\sigma_1=\cos\theta\,dx+\sin\theta\,dy,\, \sigma_2=d\theta, \, \sigma_3=-\sin\theta\,dx+\cos\theta\,dy.\]
We note that in the coordinate system $(t,x,y,\theta)$ on $M=E(2)\times_{SO(2)}\mathbb{R}^2$,
with $t\in(-\infty,\eta)$, a curve will leave every compact set only if along it, either $x$ or $y$
approach $\pm\infty$ or $t$ approaches $\eta$ (Proposition~\ref{prop:bolt} is used to make this
statement). Now
a curve of finite length $\gamma:I\to M$ has length
\[ L(\gamma) = \int_I |\gamma'(u)|\,du\geq \int_I |g(\gamma'(u),v)|\,du\]
for any unit vector field $v$.
Employing the metric in the form \Ref{bianchiAmet}, and choosing unit vector fields in the
directions of the frame fields $\{\partial_t, X_1, X_3\}$ of \Ref{X123}, we have
\begin{align}
\label{length1} L(\gamma) &\ge \int_I abc\left|t'(u)\right|du\ge\left|\int_{t(I)} abc\,dt \right|, \\
\label{length2} L(\gamma) &\ge \int_I a\left|\cos(\theta(u))x'(u)+\sin(\theta(u))y'(u)\right|du \\
\nonumber &\ge q\int_I \left|\cos(\theta(u))x'(u)+\sin(\theta(u))y'(u)\right|du, \\
%\label{length3} L(\gamma) &\ge \int_I b\left|\theta'(u)\right|du, \\
\label{length4} L(\gamma) &\ge \int_I c\left|-\sin(\theta(u))x'(u)+\cos(\theta(u))y'(u)\right|du\\
\nonumber &\ge q\int_I \left|-\sin(\theta(u))x'(u)+\cos(\theta(u))y'(u)\right|du.
\end{align}
Now equations \Ref{eta_dist} and \Ref{length1} imply that along $\gamma$, $t$ is bounded away from $\eta$.
Then, as $|\cos(\theta(u))|\le 1$ and $|\sin(\theta(u))|\le 1$, equations \Ref{length2} and \Ref{length4} give
\begin{align*} L(\gamma)&\ge q\int_I \left|\cos(\theta(u))x'(u)+\sin(\theta(u))y'(u)\right||\cos(\theta(u))|du,\\
L(\gamma)&\ge q\left|\int_I [\cos^2(\theta(u))x'(u)+\cos(\theta(u))\sin(\theta(u))y'(u)]du\right|, \\
L(\gamma)&\ge q\int_I [\cos^2(\theta(u))x'(u)+\cos(\theta(u))\sin(\theta(u))y'(u)]du\ge-L(\gamma). \\
\end{align*}
Similarly,
\[L(\gamma)\ge q\int_I [\sin^2(\theta(u))x'(u)-\cos(\theta(u))\sin(\theta(u))y'(u)]du\ge-L(\gamma). \]
Summing these we get
\[ q\left|\int_I x'(u)du\right|\le 2L(\gamma), \]
showing that $x$ is bounded along $\gamma$.  A similar calculation shows that $y$ is bounded.
Therefore $g$ is a complete metric on $M$.
\end{proof}

\section{Incompleteness in the Ricci flat case with $N=0$}\lb{generalized}

As mentioned in the comments of section \ref{OdE}, there exists a ``gauge" freedom of
rotating $\xx$, $\yy$ in their span, which allows one to simplify the systems in Theorem~\ref{ODE-thm}. We now employ this for case (2) of that theorem. It can easily be
shown that in the case of rotation by angle $\theta$ in $\HH$, $S$ transforms to $S+2\theta$.
In this section and the next one we make the pointwise choice of the function $\theta$ that
results in $S=\pi/4$. With that choice we have the following rephrasing, {\em but also strengthening}, of part $(2)$ of Theorem~\ref{ODE-thm}.
%In the appendix we show that Lie bracket conditions \Ref{brack1}-\Ref{rels2}, together with
%the generalized PDEs \Ref{KE-eqns} characterizing the K\"ahler-Einstein condition, reduce,
%when $\lam=0$, to a locally defined system of four generalized PDEs in the plane.
%Our purpose in this section is to introduce this system, and describe an equivalent geometric %condition.

\begin{thm}\lb{PDE-thm}
%{\bf [second version of Theorem 2 (2)]}
%\begin{namedthm}{Theorem~2,~part~(2),~second~version}\lb{PDE-thm}
%\noindent
%{\bf Theorem~2,~part~(2),~second~version}
%{\it
Let $(M,g)$ be a Riemannian 4-manifold admitting an orthonormal frame
$\{\kk,\tT,\xx,\yy\}$ as in section 3, satisfying in particular
relations \Ref{brack1}-\Ref{rels2} with $N=0$. Let $V$ be the domain
of $\ta$ of \Ref{grad}. Then in the set $V\cap\{F+G\ne 0\}$, equations
\Ref{KE-eqns} hold with $\lam=0$ if and only if,
after a pointwise rotation of $\xx$, $\yy$ by an appropriate angle,
$P$, $Q$, $R$, $S$, $L$ of \Ref{chan-var} satisfy the system
\begin{align}
i&)\ S=\pi/4,\nonumber\\
ii&)\ d_{\kk-\tT}R=R(P+2L),  &&iii)\ d_{\kk+\tT}R=-RQ,\nonumber\\
iv&)\ d_{\kk-\tT}L=2L^2+R^2/2, &&\ v)\ d_{\kk-\tT}P-d_{\kk+\tT}Q=2LP+2R^2,\lb{sys2}
\end{align}
and all $d_\xx$, $d_\yy$ derivatives of these functions vanish.
\end{thm}

The proof of this theorem is contained in the appendix. We will also need the
following lemma.
\begin{lemma}\lb{scnd-order}
Given vector fields $\kk-\tT$, $\kk+\tT$, and smooth functions $R$, $L$ on a given manifold,
equations \Ref{sys2}ii)-v) hold for some smooth functions $P$, $Q$, and away from the zeros of $R$,
if and only if \Ref{sys2}iv)
and
\begin{equation}
\lb{r2ord} 3R^2 = d_{\kv-\tv}^2 \log |R| + d_{\kv+\tv}^2\log |R| - 2Ld_{\kv-\tv}\log |R|
\end{equation}
hold there.
\end{lemma}
\begin{proof}
Away from the zeros of $R$, equation \Ref{r2ord} is a direct consequence of \Ref{sys2}ii)-v).
Conversely, given \Ref{sys2}iv), define $P$ and $Q$ so that \Ref{sys2}ii),iii) both hold
whenever $R$ is nonzero. Then one easily verifies that \Ref{sys2}v) follows from
\Ref{r2ord} and \Ref{sys2}iv).
\end{proof}
%}
%\end{namedthm}

%Note also that one can similarly use
%such a rotation to simplify the form of equation \Ref{sys2}, effectively eliminating in this case %one of the variables $S$, $P$ and $Q$.

\subsection{Laplace Operator}

\begin{prop}\lb{r-laplace}
For $(M,g)$ as in Theorem~\ref{PDE-thm} with functions $P,Q,R,S,L$ satisfying \Ref{sys2}, $R$ will satisfy away from its zeros
\[ \Delta \log |R| = \frac{3}{2}R^2, \]
where $\Delta$ is the Laplacian of $g$.
\end{prop}
\begin{proof}
For a function $u$ satisfying $d_\xv u=d_\yv u=0$,
\[\nabla u = (d_\kv u)\kv + (d_\tv u)\tv \]
Then using $\mathcal{L}_\mathbf{v}d\mathrm{vol}=(\mathrm{div}\,\mathbf{v})d\mathrm{vol}$
and \Ref{d-frame} in the appendix, we get
\begin{align*} \mathcal{L}_{\nabla u} d\mathrm{vol} &= d(\iota_{\nabla u}d\mathrm{vol}) \\
&=d((d_\kv u) \tf\wedge\xf\wedge\yf - (d_\tv u)\kf\wedge\xf\wedge\yf) \\
&=[d_\kv^2 u+(d_\kv u)(-L-A-D)+d_\tv^2 u-(d_\tv u)(-L+E+H)]d\mathrm{vol} \\
&=[d_\kv^2 u + d_\tv^2 u - (d_{\kv-\tv}u)(L+N)]d\mathrm{vol}
\end{align*}
showing
\[ \Delta u = d_\kv^2 u + d_\tv^2 u - (d_{\kv-\tv}u)(L+N) \]

But $N=0$, and Lemma~\ref{scnd-order} implies that $R$ satisfies away from its zeros the second order equation \Ref{r2ord}.
\end{proof}

\subsection{2d Leaf Metrics}\lb{leaf2d}
Since $\VV=\mathrm{span}\{\kv,\tv\}$ is integrable, there is a foliation of $M$ with leaves tangent to $\VV$.
It is easily checked using \Ref{conn} that the leaves are totally geodesic.
Abusing notation, let $\kf$ and $\tf$ now be the forms pulled back to a leaf in this foliation.
Then the metric induced on the leaf is just $\bar g = \kf^2+\tf^2$. To compute its Gauss curvature,
note that
\[ d\left(\begin{array}{c} \kf \\ \tf \end{array}\right)=-\left(\begin{array}{c} L\kf\wedge\tf \\ L\kf\wedge\tf \end{array}\right) =
 -\left(\begin{array}{cc} 0 & L(\kf+\tf) \\ -L(\kf+\tf) & 0 \end{array}\right)\wedge\left(\begin{array}{c} \kf \\ \tf \end{array}\right) \]
while \Ref{d-frame} in the appendix implies
\[ d(L(\kf+\tf)) = dL\wedge(\kf+\tf)-2L^2\kf\wedge\tf = (d_{\kv-\tv}L-2L^2)\kf\wedge\tf, \]
showing that by \Ref{sys2}iv) the Gauss curvature of $\bar g$ is
\be\lb{GaussC}
K_{\bar{g}} = \frac{R^2}{2}.
\end{equation}
This also gives the following corollary.
\begin{cor}
A vertical leaf $(\Sigma,\bar{g})$ has Gauss curvature $K_{\bar{g}}\ge0$ satisfying
\begin{equation}\label{gauss_eq} \bar{\Delta} \log K_{\bar{g}} = 6K_{\bar{g}} \end{equation}
away from its zeros, where $\bar{\Delta}$ is the Laplacian of $\bar{g}$.
\end{cor}
\begin{proof}
This holds because the leaf Laplacian is
\be\lb{Lap-leaf} \bar\Delta u = d_\kv^2 u + d_\tv^2 u - (d_{\kv-\tv}u)L. \end{equation}
\end{proof}
\subsubsection{Prescribed Gauss Curvature}
The conformal metric
\[e^{2u}\bar g = e^{2u}(\kf^2+\tf^2)\]
has Gauss curvature,
\begin{equation} \label{pre_gauss} K_{e^{2u}\bar g} = e^{-2u}(K_{\bar{g}}-\bar{\Delta} u). \end{equation}
\begin{cor}\lb{gau-conf} Whenever defined, the conformal metric $K_{\bar{g}}\bar{g}$ has constant curvature -2.
\end{cor}
\begin{proof}
This follows immediately from  the prescribed Gauss curvature equation \Ref{pre_gauss} with $u=\frac{1}{2}\log K_{\bar{g}}$.
\end{proof}

\begin{prop}
Suppose that $(\Sigma,\bar{g})$ is a surface of non-negative Gauss curvature satisfying \Ref{gauss_eq}, then
$\Sigma$ is not a complete non-flat finitely connected surface.
\end{prop}
\begin{proof}
Suppose that $(\Sigma,\bar{g})$ is a finitely connected non-flat complete surface with $K_{\bar{g}}$ non-negative.
By a result of Cecchini in \cite{c} it is also integrable on $\tilde{\Sigma}$, and
\be\lb{Cech} \int_{\Sigma} K_{\bar{g}} dA_{\bar{g}} \le 2\pi. \end{equation}
But using \Ref{gauss_eq}, a result of Yau  given in \cite[Thm.~1]{y} implies $K_{\bar{g}}$ is not integrable, which contradicts \Ref{Cech}.
\end{proof}

The same argument would apply to any completion of $\Sigma$ by attaching a single point (corresponding to attaching a bolt to $M$ transverse to $\VV$).
Since the leaves are totally geodesic, if they are non-flat and finitely connected the K\"ahler metric on $M$ is not complete.

%Since $K_{\bar{g}}$ is non-negative, by equation (\ref{gauss_eq}) away from its zeros $\log %K_{\bar{g}}$ is subharmonic.
%Since $\exp(t/2)$ is a convex increasing function,
%\[ \exp\left(\frac{1}{2}\log K_{\bar{g}}\right) = \sqrt{K_{\bar{g}}} \]
%is also subharmonic there.
%Now suppose to the contrary that  $(\tilde\Sigma,\tilde{g})$ is a finitely connected non-flat %complete surface obtained by adding a point to $\Sigma$.
%Since $K_{\tilde{g}}$ is smooth it must be non-negative and $\sqrt{K_{\tilde{g}}}$ and must be %subharmonic on $\tilde\Sigma$.
%By a result of Cecchini in \cite{c} we have
%\[ \int_{\tilde{\Sigma}} K_{\tilde{g}} dA_{\tilde{g}} \le 2\pi. \]
%Therefore $\sqrt{K_{\tilde{g}}}$ is a subharmonic function in $L^2(\tilde\Sigma)$,
%and by a result of Yau in \cite{y} this implies that $K_{\tilde{g}}$ is constant.
%Then $K_{\bar{g}}$ is also constant, and the only constant solution of (\ref{gauss_eq}) is %$K_{\bar{g}}=0$.
%This gives a contradiction.
%\end{proof}

\section{Local Ricci flat metrics with $N=0$}\lb{coords}

\subsection{Existence of leaf-like metrics}\lb{leaf-ex}
In subsection \ref{leaf2d}, two properties that a leaf metric was shown to have, were equation \Ref{gauss_eq} for the Gauss curvature, and Corollary~\ref{gau-conf}. We now construct metrics with these properties, which we call leaf-like metrics.
%is determined by $R$ and $L$ solving \Ref{sys2}iv).
%To do so, we examine these properties from the viewpoint of the conformally related metric.

On an open set in $\mathbb{R}^2$, let $g_0$ be a flat metric, $\ell$ a positive function
such that $\tilde{g}:=\ell g_0$ is a hyperbolic metric of constant curvature $-2$, and $h$ a harmonic function. Define
\[
\bar{g}=\ell^{-1/2}e^{-h} g_0.
\]
This metric has Gauss curvature
\[
K_{\bar{g}}=\frac{\Delta_0(\log(\ell^{-1/2}e^{-h}))}{-2\ell^{-1/2}e^{-h}}=
-\frac12 e^h\ell^{1/2}\left(\frac{-\Delta_0\log\ell}2\right)=
-\frac12 e^h\ell^{1/2}(-2\ell)=e^h\ell^{3/2}.
\]
Thus $K_{\bar{g}}>0$ and $K_{\bar{g}}\bar{g}=\tilde{g}$ is the hyperbolic metric.
On the other hand the standard formula relating the Laplacians
of $\bar{g}$ and $g_0$ gives
\[
\bar{\Delta}\log K_{\bar{g}}=\ell^{1/2}e^h\Delta_0(\log(e^h\ell^{3/2}))
=\fr 32\ell^{1/2}e^h\Delta_0(\log\ell)=\fr 32\ell^{1/2}e^h(4\ell)=6K_{\bar{g}}.
\]

Our final theorem is as follows.
\begin{thm}\lb{g-bar-g}
Given a leaf-like metric $\bar{g}=\bar{g}(l,h,g_0)$, there exist a four-dimensional
local K\"ahler Ricci-flat metric $g$ with a totally-geodesic two-dimensional foliation
whose leaves are isometric to $\bar{g}$. The dimension of the Lie algebra of
Killing fields of $g$ is at least $2$.
\end{thm}
The proof of this theorem will be given in the rest of this section.
Note that we have already shown examples of such metrics in
section 8, with a $3$-dimensional Lie algebra of Killing fields.
However for non-trivial harmonic functions $h$, or more precisely
whenever $\ell^{-1/2}e^{-h}$ depends non-trivially on two coordinate functions,
the dimension will drop to $2$.

%$\sqrt{\det\bar g}^{-1} \bar{g}^{ij}\partial_i(\sqrt{\det\bar g}\partial_jf)$

%Denote, as before, the Gauss curvature of a leaf metric
%$\bar{g}$ by $K_{\bar{g}}$. The hyperbolic metric $\tilde{g}=K_{\bar{g}}\bar{g}$
%of Gauss curvature $-2$ satisfies the analog of \Ref{pre_gauss},
%namely
%\[
%K_{\bar{g}}=K_{K_{\bar{g}}^{-1}\tilde{g}}=K_{\bar{g}}(-2-\tilde{\Delta}((1/2)\log %K_{\bar{g}}^{-1})),
%\]
%which easily simplifies away from the zeros of $K_{\bar{g}}$ to
%\[
%\tilde{\Delta}\log K_{\bar{g}}=6.
%\]
%Solutions of this equation may be obtained as follows. $\tilde{g}$ is conformally
%flat, so let $\tilde{g}=\ell g_0$, for a smooth positive function $\ell$ and a
%flat metric $g_0$, which we take to be the
%standard Euclidean metric with Laplacian $\Delta_0$. Then, from the
%standard formulas for the Gauss curvature, and the relation between Laplacians,
%we have
%\begin{align*}
%-2=K_{\tilde{g}}&=-\Delta_0(\log\ell)/2\ell,\\
%6=\tilde{\Delta}\log K_{\bar{g}}&=\Delta_0(\log K_{\bar{g}})/\ell
%\end{align*}
%Combining these, we see that
%\[
%\Delta_0\log (K_{\bar{g}}/\ell^{3/2})=0
%\]
%so that $K_{\bar{g}}=R^2/2=\ell^{3/2}e^h$ for a harmonic function $h$, and
%\[
%\bar{g}=\ell^{-1/2}e^{-h} g_0.
%\]
%Thus on an open set in $\mathbb{R}^2$, a flat metric $g_0$, a harmonic function $h$
%and a conformal factor $\ell$ relating $g_0$ and a hyperbolic metric of Gauss curvature $-2$
%give rise a leaf-like metric satisfying the above-mentioned properties.
%Note for later reference that $R$ is nowhere zero in this construction.

\subsection{Coordinates for a leaf-like metric leading to equations \Ref{sys2}}\lb{leaf-coord}
We introduce a coordinate representation for a given leaf metric in preparation for exhibiting
a coordinate representation for $g$.

Let $\bar{g}$ be a leaf-like metric with Gauss curvature $K_{\bar{g}}$ on
an open set in $\mathbb{R}^2$. Since $K_{\bar{g}}>0$, the metric $\bar{g}$ is locally embeddable
in $\mathbb{R}^3$ and hence admits geodesic parallel coordinates.
Choose a ``homothetic" version $x$, $y$ of these coordinates with domain $U_1$, in which the
metric takes the form:
\[
\bar{g}=2(dx^2+c^2dy^2)
\]
for some nowhere vanishing $c=c(x,y)$.
Then $K_{\bar{g}}=-c_{xx}/c$, so define
\be\lb{R-c} R:=\sqrt{-2c_{xx}/c}. \end{equation}
%Thus $c$ solves
%\[
%c_{xx}+K_{\bar{g}} c=0,\qquad c(0,y)=f(y),\qquad c_x(0,y)=k(y).
%\]
%for some given smooth functions $f(y)$, $k(y)$ with $f$ nowhere vanishing.
Next, define vector fields $\kk$, $\tT$ by the formulas $\kk-\tT:=\partial_x$, $\kk+\tT:=\fr1c\partial_y$, so that
$\bar{g}=2(\widehat{\kk-\tT}^2+\widehat{\kk+\tT}^2)=\hat\kk^2+\hat\tT^2$.

Now let $L:=-c_x/(2c)$. It is then easy to check that
[\kk,\tT]=L(\kk+\tT), and equation \Ref{sys2}iv) holds. Furthermore, as
$R^2/2=K_{\bar{g}}$ solves \Ref{gauss_eq}, and one can verify
formula \Ref{Lap-leaf} for the Laplacian, it follows that equation
\Ref{r2ord} also holds in the coordinate domain $U_1$.

It thus follows from lemma \ref{scnd-order} that equations \Ref{sys2} hold
there as well, for $S=\pi/4$ and $P$, $Q$ whose formulas in terms of $c$ are

\be
%L&=-c_x/(2c), &&R=\pm\sqrt{-c_{xx}/c}, &&&S=\pi/4\nonumber\\
P=\frac 12\big(\log \fr{-c_{xx}}{c}\big)_x+\fr{c_x}c,\qquad
Q=-\fr 1{2c}\big(\log \fr{-c_{xx}}{c}\big)_y.\lb{main-var}
\end{equation}

\subsection{Coordinates for $g$}

For $P$, $Q$, $R$ as in \Ref{main-var}, \Ref{R-c}, set
\begin{align*}
\al&:=R/\sqrt{2},     &&\bet:=Q/2,\\
\nu&:=P/2+R/\sqrt{2}, &&\chi:=-P/2+R/\sqrt{2},\\
%\mu&:=A+E=0,\qquad \eta:=D+H=0.
\end{align*}
which are smooth functions in $U_1$.

Let $U_2$ be an open set in the plane with coordinates $u$, $v$
so that $M:=U_1\times U_2$ has coordinates $x$, $y$, $u$, $v$.
Define on $M$ vector fields
\begin{align*}
&\kk-\tT=\partial_x, \qquad\kk+\tT=(1/c)\partial_y,\\
&\xx=a\partial_u+b\partial_v, \qquad \yy=r\partial_u+s\partial_v,
\end{align*}
where $c=c(x,y)$ is as in subsection \ref{leaf-coord},
and $a$, $b$, $r$, $s$ are functions of $x$, $y$ which are
solutions to the system
\begin{align}
&a_x=\al a+\bet r, &&r_x=-\bet a-\al r,\nonumber\\
&b_x=\al b+\bet s, &&s_x=-\bet b-\al s,\nonumber\\
&a_y/c=\nu r, &&r_y/c=\chi a,\nonumber\\
&b_y/c=\nu s, &&s_y/c=\chi b.\lb{vec-sys}
\end{align}
We need to verify that solutions to this overdetermined elliptic system exist, and
we will do so via elementary means. Note that the eight equations in \Ref{vec-sys}
decouple in pairs, line by line, and partial integration of each pair is equivalent
to solving at most a second order linear equation in one of the two unknowns of the pair.
For example, for the first line in \Ref{vec-sys}, if $\bet$ is nowhere vanishing
the two equations resolve to $a_{xx}-(\bet_x/\beta)a_x+(\bet(\bet-\al^2)+\al\bet_x/\bet-\al_x)a=0$
and $r=(a_x+\al a)/\bet$. Regarding the former equation as an ODE for fixed $y$, its
smooth coefficients guarantee local existence. Note that $\bet=0$, which is the case where $c$ depends only on $x$, examples of which have been given in section~\ref{ric-flat}, is even simpler
as $a$ and $r$ completely decouple.

Similarly the pair in the third line in \Ref{vec-sys} translates to $a_{yy}-(c\nu)_y/(c\nu)a_y-c^2\chi\nu a=0$ and $r=a_y/(c\nu)$, with a similar remark as above if $\nu$ vanishes
%(but note that $\nu$ and $\chi$ cannot both vanish as $R>0$).
and local existence as an ODE in $y$ is again guaranteed.

The integrability conditions ensuring that the above partial integrations give consistent
solutions to these PDEs are all guaranteed to hold via \Ref{sys2}.
For example, from the first and third lines of
\Ref{vec-sys}, $a_{xy}=(\al_y+c\bet\chi)a+(\bet_y+c\al\nu)r$ and
$a_{yx}=-c\bet\nu a+((c\nu)_x-c\al\nu)r$, so that the mixed partials of $a$
will be equal if
\begin{align*}
\al_y+c\bet\chi=-c\bet\nu,\\
\bet_y+c\al\nu=(c\nu)_x-c\al\nu.
\end{align*}
Unraveling the definitions of these quantities,
the first equation reduces to \Ref{sys2}iii), while
the second equation holds in lieu of \Ref{sys2}ii),v).

We now intend to apply Theorem~\ref{PDE-thm}, so we need to verify
the Lie bracket relations \Ref{brack1}-\Ref{rels2}. Clearly
$\xx$ and $\yy$ commute, whereas $[\kk,\tT]$ was already discussed.
The remaining four relations follow from \Ref{vec-sys}, so long
as the functions $A,\ldots H$ are defined by $\al=A-E=H-D$,
$\bet=B-F=G-C$, $\nu=B+F$, $\chi=C+G$ and $0=A+E=D+H$. It is a routine
matter to check that under these definitions \Ref{rels1}-\Ref{rels2}
also hold, and $P$, $Q$, $R$, $S$ are derived from $A,\ldots H$
in accordance with \Ref{chan-var}. Finally, the domain of $\ta$
is clearly that of $x$, i.e. all of $M$.

By Theorem~\ref{PDE-thm}, \Ref{KE-eqns} holds, and as $N=0$, the metric
\begin{multline*}
g=\hat{\xx}^2+\hat{\yy}^2+2(\widehat{\kk-\tT}^2+\widehat{\kk+\tT}^2)\\
=\fr1{(as-rb)^2}\left((s\,du-r\,dv)^2+(-b\,du+a\,dv)^2\right)+2(dx^2+c^2dy^2).
\end{multline*}
is K\"ahler and Ricci-flat, and has a totally geodesic foliation with leaf metric
$\bar{g}$. Since its coefficients depend on at most two coordinates, the dimension
of the Lie algebra of Killing fields is at least two.
This concludes the proof of Theorem~\ref{g-bar-g}.

\appendix

\section{Outline of the derivation of the ODE and PDE systems}
\subsection{Generalized PDEs}
Suppose one is given a $4$-manifold with a frame $\kk$, $\tT$, $\xx$, $\yy$ satisfying
the Lie bracket relations \Ref{brack1}-\Ref{brack3} for functions $A$, $B$, $C$, $D$, $E$, $F$, $G$, $H$, $L$, $N$ on the frame domain. The dual coframe
$\hat\kk$, $\hat\tT$, $\hat\xx$, $\hat\yy$ then satisfies
\begin{align}
d\kf&=-N\xf\wedge\yf-L\kf\wedge\tf,\nonumber\\
d\tf&=-N\xf\wedge\yf-L\kf\wedge\tf,\nonumber\\
d\xf&=-A\kf\wedge\xf-C\kf\wedge\yf-E\tf\wedge\xf-G\tf\wedge\yf,\nonumber\\
d\yf&=-B\kf\wedge\xf-D\kf\wedge\yf-F\tf\wedge\xf-H\tf\wedge\yf.\lb{d-frame}
\end{align}
The vanishing of $d^2$ on the coframe $1$-forms gives four equations, two of which
are identical. Writing, for example, $dN=d_\kk N\kf+d_\tT N\tf+d_\xx N\xf+d_\yy N\yf$ etc.
and separating components yields $12$ scalar equations
\begin{align}
d_\xv L&=0,\qquad d_\yv L=0,\nonumber\\
d_\yv A&=d_\xv C,\qquad d_\yv B=d_\xv D,\qquad d_\yv E=d_\xv G,\qquad d_\yv F=d_\xv H,\nonumber\\
d_\tv N&=NE+NH+LN,\qquad d_\kv N=NA+ND-LN,\lb{nl}\\
d_\tv A&=d_\kv E-AL+CF-EL-GB,\lb{dta} \\
d_\tv B&=d_\kv F-BL+BE+DF-FL-FA-HB,\lb{dtb} \\
d_\tv C&=d_\kv G+AG-CL+CH-EC-GL-GD,\lb{dtc} \\
d_\tv D&=d_\kv H+BG-DL-FC-HL.\lb{dtd}
\end{align}
Adding and subtracting the two equations \Ref{nl}, the two equations \Ref{dta} and \Ref{dtd}
and the two equations \Ref{dtb}-\Ref{dtc}, while using relations \Ref{rels1}-\Ref{rels2},
yields six equations of which only five are independent. The resulting equivalent system is
\begin{align}
&d_\xv L=0,\qquad d_\yv L=0,\lb{ll}\\
&d_\yv A=d_\xv C,\qquad d_\yv B=d_\xv D,\qquad d_\yv E=d_\xv G,\qquad d_\yv F=d_\xv H,\lb{foursome}\\
&d_{\kv+\tv} N=0,\qquad d_{\kv-\tv} N=2N^2-2LN,\lb{nll} \\
&d_\tv (F+G)=-d_\kv (B+C)-(F+G)L+(B+C)L-2(F+G)B+2(B+C)F,\lb{long1}\\
&d_\kv (F+G)=d_\tv (B+C)+(B+C)L+(F+G)L+F^2-G^2+B^2-C^2,\lb{long2}\\
&d_\tv (B-C)=d_\kv (F-G)-(B-C)L-(F-G)L-(B+C)^2-(F+G)^2.\lb{long3}
\end{align}
Assume now that $M$ admits a K\"ahler metric making our frame orthonormal,
which is additionally Einstein. Then, in addition to the above system,
the six equations \Ref{KE-eqns} reproduced below also hold.
\begin{align}
\lam&=-N(2L+C-H+A-F),\lb{lam1}\\
\lam&=-L(2L+C-H+A-F)+d_{\kv-\tv}L-d_\tv(C-H)+d_\kv(A-F),\lb{lam2}\\
0&=d_\xv(L+C-H),\qquad 0=d_\xv(L+A-F),\lb{other1}\\
0&=d_\yv(L+C-H),\qquad 0=d_\yv(L+A-F).\lb{other2}
\end{align}
At this point our derivation splits into cases.

\subsection{The case $\lam\ne 0$}
If the Einstein constant $\lam$ is nonzero, then by \Ref{lam1}
%We first assume
\be\lb{generic}
\text{$2L+C-H+A-F$ is nowhere vanishing}.
\end{equation}
In that case, \Ref{lam1}, \Ref{other1} and \Ref{other2}
clearly imply
\begin{align}
d_\xx N=0,\qquad d_\yy N=0.\lb{NN}
\end{align}
Additionally, by the second of equations \Ref{brack1}, we have
the following {\em basic fact}: for any
smooth function $f$ on the frame domain,
\[
\text{if $d_\xx f=d_\yy f=0$ then $d_{\kk+\tT}f=0$.}
\]
Thus from \Ref{ll} $d_{\kk+\tT}L=0$. This, in conjunction with
\Ref{ll}, \Ref{other1}, \Ref{other2} and the {\em basic fact} imply in turn
\begin{align}
d_\xv (A-F)&=0,\qquad d_\yv (A-F)=0,\qquad d_{\kv+\tv}(A-F)=0,\nonumber\\
d_\xv (C-H)&=0,\qquad d_\yv (C-H)=0,\qquad d_{\kv+\tv}(C-H)=0.\lb{leftover}
\end{align}
Among these, the third and sixth equations imply
\be\lb{switch}
d_{\kk-\tT}(2L+C-H+A-F)=2(d_{\kk-\tT}L-d_\tT(C-H)+d_\kk(A-F)).
\end{equation}
As we can substitute this in \Ref{lam2} we note the following:
\Ref{lam1} implies both that $N$ is nowhere vanishing, and that
we can replace $2L+C-H+A-F$ with $-\lam/N$. This
then implies that the second equation in \Ref{nll} yields \Ref{lam2}.
We can thus drop \Ref{lam2} from our system.

We now introduce the change of variables \Ref{chan-var} valid at points
of the frame domain where
\[
F+G\ne 0,
\]
and note its inverse.
\begin{align}
B&=[(P+Q)+2R\sin S]/4,\qquad C=[-(P+Q)+2R\sin S]/4,\nonumber\\
F&=[(P-Q)+2R\cos S]/4,\qquad G=[-(P-Q)+2R\cos S]/4.\lb{change}
\end{align}
Since
\be\lb{intermediate}
A-F=(N-F+G)/2,\qquad C-H=(N-B+C)/2
\end{equation}
by \Ref{rels1}-\Ref{rels2}, it follows from \Ref{leftover} and \Ref{NN} that
\be\lb{BCFG}
\text{$d_\xx (F-G)=0$, $d_\yy (F-G)=0$, $d_\xx (B-C)=0$ and $d_\yy (B-C)=0$.}
\end{equation}
Hence
\begin{align*}
d_\xx P&=0,\qquad d_\yy P=0,\qquad d_{\kk+\tT}P=0,\\
d_\xx Q&=0,\qquad d_\yy Q=0,\qquad d_{\kk+\tT}Q=0,
\end{align*}
where the above {\em basic fact} was also employed.

Next we show that the four equations \Ref{foursome} can be replaced by two
equivalent equations in terms of $R$ and $S$. On the one hand,
by \Ref{BCFG}, $d_\yy B=d_\yy(B+C)/2=d_\yy(R\sin S)/2$, $d_\xx C=d_\xx (B+C)/2=d_\xx (R\sin S)/2$,
$d_\yy F=d_\yy(F+G)/2=d_\yy(R\cos S)/2$, $d_\xx G=d_\xx(F+G)/2=d_\xx(R\cos S)/2$.
On the other hand $d_\yy A = d_\yy(R\cos S)/2$, $d_\xx D = -d_\xx(R\cos S)/2$,
$d_\yy E = -d_\yy(R\sin S)/2$, $d_\xx H = d_\xx(R\sin S)/2$, because \Ref{rels1}-\Ref{rels2}
imply $A=(N+F+G)/2$, $D=(N-F-G)/2$, $E=-(N+B+C)/2$, $H=(-N+B+C)/2$.
Thus \Ref{foursome} can be replaced by the two equations
\begin{align}
d_\yv (R\cos S)&=d_\xv (R\sin S),\nonumber\\
d_\yv (R\sin S)&=-d_\xv (R\cos S).\lb{cs-sn}
\end{align}

Finally, equations \Ref{long1}-\Ref{long3} are replaced by
\begin{align}
d_\tv (R\cos S)&=-d_\kv (R\sin S)-RL(\cos S-\sin S)\nonumber\\
&+\frac{1}{2}(P-Q)R\sin S-\frac{1}{2}(P+Q)R\cos S,
\lb{trig1}\\
d_\kv (R\cos S)&=d_\tv (R\sin S)+RL(\sin S+\cos S)\nonumber\\
&+\frac{1}{2}(P-Q)R\cos S+\frac{1}{2}(P+Q)R\sin S,
\lb{trig2}\\
\frac{1}{2}d_\tv (P+Q) &= \frac{1}{2}d_\kv (P-Q)-PL-R^2.\lb{longy}
\end{align}
The justification is straightforward, except for noting that the last two terms of
\Ref{long1} equal $2(CF-GB)$, which, via \Ref{change},
is calculated to equal $[R(P-Q)\sin S-R(P+Q)\cos S]/2$.

So far, we know that the $d_\xx$, $d_\yy$, $d_{\kk+\tT}$ derivatives
of $L$, $N$, $P$ and $Q$ vanish. Our goal now is to prove
the same for $R$ and $S$, and also find the $d_{\kk-\tT}$
derivatives of all these quantities. Looking at \Ref{longy},
as $d_{\kk+\tT}Q=0$, it becomes
\be\lb{PP}
\fr 12d_{\kk-\tT}P=PL+R^2.
\end{equation}
Applying $d_\xv$, $d_\yv$, and $d_{\kv+\tv}$ to this equation, and employing the Lie bracket relations
\Ref{brack1}-\Ref{brack3}, we find that
\begin{align*}
d_\xv R&=0,\qquad d_\yv R=0,\qquad d_{\kv+\tv}R=0.
\end{align*}
Using this we see from \Ref{cs-sn} that, as $R\ne 0$ under our assumptions, we have
\begin{align*}
d_\xv S&=0,\qquad d_\yv S=0,\qquad d_{\kv+\tv}S=0.
\end{align*}

Now \Ref{trig1}-\Ref{trig2} can be converted to the form
\begin{align*}
\al=-\beta\tan S,\qquad \beta=\alpha\tan S,
\end{align*}
for certain expressions $\al$, $\beta$ which thus vanish. Their
vanishing is equivalent to the equations
\begin{align}\lb{R-S}
d_\tv R &= -R d_\kv S - RL - \frac{1}{2}(P+Q)R,\qquad
d_\kv R = R d_\tv S + RL + \frac{1}{2}(P-Q)R.
\end{align}
As $d_{\kv+\tv}R=0$ and $d_{\kv+\tv}S=0$, adding and subtracting these
to the above gives
\be\lb{SR}
d_{\kv-\tv}S = -Q,\qquad d_{\kv-\tv} R =  2RL + PR.
\end{equation}
Now from \Ref{grad} we have $d\tau=\kf-\tf$ in an open set $V$, and
$d_{\kk-\tT}\ta=2$ due to the metric values on $\kk$ and $\tT$.
Thus $L$, $N$, $P$, $Q$, $R$ and $S$ are functions of $\tau$ in the sense of section \ref{OdE}.
Employing the abuse of notation described there, the second of equations \Ref{nll}
along with equations \Ref{PP}  and \Ref{SR} show that ODEs in \Ref{sys} hold for
$N'$, $P'$, $R'$ and $S'$.

As $N$ is nowhere vanishing, the only remaining independent
equation is \Ref{lam1}, which in the variables \Ref{chan-var} is written,
with the help of the definition of $P$ and \Ref{rels1}-\Ref{rels2}, in the
form
\be\lb{lam-end}
2\lam=-N(4L+2N-P).
\end{equation}
However, differentiating \Ref{lam-end} with respect to $\ta$, then replacing $N'$
by its expression from \Ref{sys} and simplifying yields, as $N$ is nowhere
vanishing,  the equation for $L'$ in \Ref{sys}. This concludes the case
$\lam\ne 0$.

\subsection{The case $\lam=0$, $N=0$}
If $\lam=0$, by $\Ref{lam1}$, at each point either \Ref{generic} does not hold,
or $N=0$. We assume the latter everywhere:
\[
\text{$N=0$ on $V\cap\{F+G\ne 0\}$.}
\]
Using \Ref{intermediate}, which still holds, when $N=0$  equation \Ref{lam2} can be written in the form
\[
-L(2L-P/2)+d_{\kk-\tT}L+d_\tT(B-C)/2-d_\kk(F-G)/2=0.
\]
Applying \Ref{long3} and \Ref{chan-var} gives \Ref{sys2a}iii).
Applying $d_\xx$ and $d_\yy$ to the latter and using \Ref{ll} gives $d_\xx R=0$, $d_\yy R=0$.
As the passage from \Ref{foursome} to \Ref{cs-sn} is still valid, the latter gives the vanishing
of $d_\xx S$ and $d_\yy S$. As the $d_\xx$, $d_\yy$ parts of \Ref{leftover} still hold,
they lead as before to the vanishing of the $d_\xx$ and $d_\yy$ derivatives of $P$ and $Q$.

Of equations \Ref{ll}-\Ref{other2}, that always hold, the ones whose consequences have not yet
been explored are \Ref{long1}-\Ref{long3}, which translate in the variables \Ref{chan-var} to
\Ref{trig1}-\Ref{longy}. Of the latter equations, the first two translate as before
to \Ref{R-S}, and adding and subtracting these gives \Ref{sys2a}i),ii).
Whereas \Ref{longy} is equivalent \Ref{sys2a}iv).

After the rotation in $\HH$ that gives $S=\pi/4$, the equations of \Ref{sys2a} turn into
those of \Ref{sys2}. Conversely, on $(M,g)$ with \Ref{brack1}-\Ref{rels2} and $N=0$,
starting from \Ref{sys2} or its general $S$ form \Ref{sys2a}, along with the assumption on the
vanishing of $d_\xx$ and $d_\yy$ on $P,Q,R,S,L$, one easily checks that the above steps are reversible and lead to equations \Ref{lam1}-\Ref{other2} with $\lam=0$, i.e. to \Ref{KE-eqns}, so that the metric is Ricci flat.

\end{document}